\documentclass[a4paper,10pt]{article}
\usepackage[utf8]{inputenc}
\usepackage[T1]{fontenc}
\usepackage[english]{babel}
\usepackage{amsmath,amsthm,amssymb}
\usepackage{url}
\usepackage{mathptmx}
\usepackage{eucal}
\usepackage{esint}
\usepackage{bm}
\usepackage{fancyhdr}
\usepackage{color}
\usepackage{tikz}
\usepackage{graphicx}
\usepackage[font=small]{caption}
\graphicspath{{Pictures/}}
\usepackage[]{hyperref}

\newtheorem{theorem}{Theorem}

\newtheorem{lemma}[theorem]{Lemma}

\theoremstyle{definition}
\newtheorem{remark}[theorem]{Remark}

\numberwithin{equation}{section}

\newcommand{\ddiv}{\operatorname{div}}
\newcommand{\Curl}{\operatorname{Curl}}
\newcommand{\rot}{\operatorname{rot}}
\newcommand{\sym}{\operatorname{sym}}

\newcommand{\tr}{\operatorname{tr}}

\newcommand{\T}{\mathcal{T}}
\newcommand{\E}{\mathcal{E}}
\newcommand{\N}{\mathcal{N}}
\newcommand{\C}{\mathbb{C}}
\newcommand{\ee}{\bm{\varepsilon}_h}
\newcommand{\osc}{\operatorname{osc}}
\newcommand{\R}{\mathbb{R}}

\DeclareMathOperator{\Res}{Res}

\newcommand{\anorm}[1]{\lvert\!\lvert\!\lvert#1 \rvert\!\rvert\!\rvert}
\begin{document}
\pagestyle{fancy}
\fancyhead{}
\setlength{\headheight}{14pt}
\renewcommand{\headrulewidth}{0pt}
\fancyhead[c]{\small \it Symmetric stress MFEM}
\title{Residual-based a~posteriori error analysis for symmetric mixed Arnold-Winther FEM}
\author{%
        C.~Carstensen\thanks{Humboldt-Universit\"at zu Berlin, Berlin, Germany
         } \and
       D.~Gallistl\thanks{Karlsruher Institut f\"ur Technologie, Karlsruhe, Germany
        }
        \and 
       J.~Gedicke\thanks{\text{Faculty~of~Mathematics,~University~of~Vienna,~Vienna,~Austria     
}}}
\date{}
\maketitle
\begin{abstract}
This paper introduces an explicit residual-based a~posteriori error analysis
for the symmetric mixed finite element method in linear elasticity
after Arnold-Winther with pointwise symmetric and  $H(\ddiv)$-conforming stress approximation.
Opposed to a previous publication, the residual-based
a posteriori error estimator of this paper is  reliable and efficient and truly explicit in that it solely 
depends  on the symmetric stress and does neither
need any additional information of some skew symmetric part of the gradient nor any efficient approximation thereof.
Hence it is straightforward to implement an  adaptive mesh-refining 
algorithm  obligatory in practical computations.

Numerical experiments verify the proven reliability and efficiency of the
new a posteriori error estimator and illustrate the improved convergence rate in comparison 
to uniform mesh-refining.
A higher convergence rates for piecewise affine data is observed in the $L^2$ stress error and reproduced
in non-smooth situations by the adaptive mesh-refining strategy. 
\end{abstract}

{\small\noindent\textbf{Keywords}
linear elasticity, mixed finite element method, a posteriori, explicit
residual-based, error estimator, reliability, efficiency,  symmetric stress finite elements,
Arnold-Winther finite element

}
%----------------------------------------------------------------------------------------------------
\section{Introduction}
%----------------------------------------------------------------------------------------------------
%----------------------------------------------------------------------------------------------------
\subsection{Overview}
\label{subsec:1.1}
%----------------------------------------------------------------------------------------------------
The design of a pointwise symmetric stress approximation $\sigma_h\in L^2(\Omega;\mathbb S)$ 
with divergence in $L^2(\Omega;\R^d)$, written $\sigma_h \in H(\ddiv,\Omega; \mathbb S)$, 
has been a long-standing challenge  \cite{A}  and the first positive examples in \cite{AW02} initiated what nowadays
is called the finite element exterior calculus \cite{AFW}.  The a posteriori error analysis of  mixed finite element methods 
in elasticity started with \cite{CD_1998} on PEERS \cite{ABD84}, where the asymmetric stress approximation $\gamma_h$ 
arises in the discretization as a Lagrange multiplier to enforce weakly the stress symmetry. This allows 
the treatment of the term $\C^{-1} \sigma_h + \gamma_h$ as an approximation of the (non symmetric) 
functional matrix $Du$ for the displacement field  \cite{CD_1998}  
with the arguments of \cite{Alonso,C97} developed for mixed
finite element schemes for a Poisson model problem. 
Here and throughout, $\C$ denotes a fourth-order elasticity tensor with two Lam\'e constants $\lambda$ and $\mu$ 
and $\C^{-1}$ is its inverse. Mixed finite element methods appear attractive in the incompressible limit
for  they typically avoid  the locking phenomenon \cite{CEG2011} as $\lambda\to\infty$.

For mixed finite element methods like the symmetric Arnold-Winther finite element schemes \cite{AW02}, 
the subtle term is the nonconforming residual: Given any piecewise 
polynomial  $\sigma_h \in H(\ddiv,\Omega; \mathbb S)$,  compute an upper bound  $\eta(\mathcal{T},\sigma_h)$ of
\[
\inf _{v\in V} \|   \C^{-1/2} \sigma_h-\C^{1/2}\varepsilon (v) \|_{L^2(\Omega)} \lesssim \eta(\mathcal{T},\sigma_h).
\]
Despite general results in this direction \cite{CC2005,CCHJ2007,ccdpas}, this task had been addressed 
only by the computation of an approximation to  the optimal $v$  with Green strain $\epsilon(v):=\sym D v$ 
or of some skew-symmetric approximation $\gamma_h$ motivated 
from the first results in  \cite{CD_1998} on PEERS. 
In fact, {\em any}  choice of a piecewise smooth and pointwise skew-symmetric $\gamma_h$  
allows for an a~posteriori error control of the symmetric stress error $\sigma-\sigma_h$ in \cite{CG_2016}. 
Its efficiency, however, depends on the 
(unknown and uncontrolled) efficiency  of the choice of $\gamma_h$ as an   approximation to the  
skew-symmetric part $\gamma$ of $Du$.

\bigskip

This paper proposes  the first reliable and efficient  {\em explicit}  residual-based a~posteriori error
estimator of the nonconforming residual  with the typical contributions to 
$\eta(\mathcal{T},\sigma_h)$ computed from  the (known) Green strain approximation 
$\ee:= \C^{-1} \sigma_h$. Besides oscillations of the applied forces in the volume and along the 
Neumann boundary, there is a volume contribution $h_T^2\|  \rot\rot\ee\|_{L^2(T)}$ for each triangle  $T\in\T$
and an edge contribution with  the jump $[\ee]_E$ across an
 interior edge $E$ with unit normal $\nu_E$, tangential unit vector $\tau_E$, and length $h_E$,
namely
 \[
  h_E^{1/2}  \| \tau_E   \cdot [\ee]_E \tau_E \|_{L^2(E)} +
    h_E^{3/2}  \| \tau_E   \cdot [\rot_{NC} \ee]_E - \partial ( \nu_E   \cdot [\ee]_E \tau_E)/\partial s   \|_{L^2(E)},
 \]
 and corresponding modification on the edges on the Dirichlet boundary with the (possibly inhomogeneous) 
 Dirichlet data; cf. Remark~\ref{remarkonjumptermvanishes} for some partial simplification of the last term
 displayed.
 
 For the ease of this presentation,  the analysis involves  explicit calculations in 
 two dimensions  without any reference to the exterior  calculus but with inhomogeneous 
 Dirichlet and Neumann boundary data.
 The main result is reliability and efficiency to control the stress error robustly in  
 the sense that the multiplicative generic constants hidden in the notation $\lesssim$ 
 do neither depend on the  (local or global) mesh-size nor on the parameter $\lambda>0$ 
 but may depend on $\mu>0$ and on the shape regularity of the underlying triangulation 
 $\mathcal{T}$ of the domain $\Omega$ 
 into triangles through a lower bound of the minimal angle therein.

%----------------------------------------------------------------------------------------------------
\subsection{Linear elastic model problem}
%----------------------------------------------------------------------------------------------------
The elastic body $\Omega$ is a simply-connected bounded Lipschitz domain $\Omega\subset \R^2$ in the 
plane  with a  (connected) polygonal  boundary  $\partial\Omega = \Gamma_D\cup\Gamma_N$ split into parts.
The displacement boundary $\Gamma_D$ is compact and of positive surface measure,
while the traction boundary is the relative open complement $\Gamma_N=\partial\Omega\backslash\Gamma_D$
with outer unit normal vector $\nu$.
Given  $u_D\in  H^1(\Omega;\R^2)$, the volume force $f\in L^2( \Omega; \mathbb{R}^2)$, and the applied surface traction
$g \in L^2( \Gamma_N; \mathbb{R}^2)$, the  linear elastic problem seeks  a displacement $u\in H^1(\Omega;\R^2)$
and a symmetric stress tensor  $\sigma\in H(\ddiv,\Omega;  \mathbb{S})$ with 
\begin{align}\label{eq:lame}
\begin{split}
	-\ddiv \sigma = f \quad\text{and}\quad  \sigma = \C\varepsilon(u)\quad\mbox{in } \Omega,\\
	u = u_D\quad\mbox{on }\Gamma_D, \qquad \sigma\nu = g\quad\mbox{ on }\Gamma_N.
\end{split}	
\end{align}
Throughout this paper, given the Lam\'e parameters 
$\lambda,\mu>0$ for isotropic linear elasticity,  the positive definite fourth-order elasticity tensor $\C$ acts as 
$\C E:=2\mu\, E+ \lambda\,\tr(E)\,1_{2\times 2} $
on any matrix $E\in\mathbb{S}$ with trace $\tr(E)$ and the $2\times 2$ unit
matrix $1_{2\times 2} $.
Note that $u_D$ acts in \eqref{eq:lame} only on $\Gamma_D$ and is an extension
of  the continuous function $u_D\in C(\Gamma_D;\R^2)$ also supposed to belong to the 
edgewise second order Sobolev space $ H^2( \E(\Gamma_D))$ below to allow second derivatives
with respect to the arc length along boundary edges. 

More essential will be a discussion on the  precise conditions on the Neumann data $g$ and its discrete 
approximation $g_h$ below for they  belong to the 
essential boundary conditions in the mixed finite element method based on the dual formulation.

In addition to the set of homogeneous displacements $V$ and the aforementioned stress space 
$H(\ddiv,\Omega;\mathbb{S}) $, namely, 
\[
V:=\{v\in H^1(\Omega;\mathbb{R}^2)\ \big\vert\  v|_{\Gamma_D}=0\}
\text{ and }
H(\ddiv,\Omega;\mathbb{S}) := \{ \tau \in L^2(\Omega;\mathbb{S}) \ \big\vert\ \ddiv\tau\in L^2(\Omega;\mathbb{R}^2)\},
\]
and with the exterior unit normal vector $\nu$ along $\partial\Omega$,
 the  inhomogeneous  stress space
\[
 \Sigma(g) := \Big\{ \sigma\in H(\ddiv,\Omega;\mathbb{S})\ \big\vert 
  \int_{\Gamma_N}\psi\cdot(\sigma\nu )\;ds = \int_{\Gamma_N}\psi\cdot g\;ds \text{ for all } \psi\in V \Big\}
\]
is defined with respect to the Neumann data $g\in L^2(\Gamma_N)$ and, in particular,   
$ \Sigma_0 :=\Sigma(0)$ abbreviates  the stress space with homogeneous 
Neumann boundary conditions.

Given data  $u_D, f,g$ as before, the dual weak formulation of \eqref{eq:lame}  seeks
$(\sigma,u)\in \Sigma(g)\times L^2(\Omega;\mathbb{R}^2)$ with
\begin{align}
  \label{eq:lame-weak}
  \begin{split}
    \int_\Omega \sigma:\C^{-1}\tau\;dx + \int_\Omega u\cdot\ddiv\tau\;dx
    &= \int_{\Gamma_D} u_D\cdot(\tau\nu)\,ds\quad\textrm{for all }\tau\in\Sigma_0,\\
    \int_\Omega v\cdot\ddiv\sigma\;dx &= -\int_\Omega f\cdot v\,dx\quad\textrm{for all } v\in L^2(\Omega;\mathbb{R}^2).
  \end{split}
\end{align}
It is well known that the two formulations are equivalent and well posed in the sense that they allow for unique 
solutions in the above spaces and are actually slightly more regular according to the reduced elliptic regularity 
theory. The reader is refereed to textbooks on finite element methods 
\cite{BrennerScott2008,Braess2007,BoffiBrezziFortin2013} for proofs and further descriptions.

Throughout this paper, the model problem considers truly mixed boundary conditions with the hypothesis
that both $\Gamma_D$ and $\Gamma_N$ have positive length. The remaining cases of a pure Neumann 
problem or a pure Dirichlet problem require standard modification and are immediately adopted. The presentation focusses
on the case of isotropic linear elasticity with constant Lam\'e parameters $\lambda$ and $\mu$ 
for brevity and many results carry over to more general situations 
(cf. Remark \ref{rematkcc1} and \ref{remarkonjumptermvanishes}  for instance). 

%----------------------------------------------------------------------------------------------------
\subsection{Mixed finite element discretization}
%----------------------------------------------------------------------------------------------------
Let $\T$ denote a shape-regular triangulation of $\Omega$ into triangles 
(in the sense of Ciarlet \cite{BrennerScott2008})
with set of nodes $\N$, set of interior edges $\E(\Omega)$, set of Dirichlet edges $\E(\Gamma_D)$
and set of Neumann edges $\E(\Gamma_N)$. The triangulation is compatible with the boundary pieces
$\Gamma_D$ and $\Gamma_N$ in that the boundary condition changes only at some vertex $\N$
and $\Gamma_D$ (resp. $\overline{\Gamma_N}$) is partitioned in $\E(\Gamma_D)$
(resp. $\E(\Gamma_N)$).

The piecewise polynomials (piecewise with respect to the triangulation
$\T$) of total degree at most $k\in\mathbb{N}_0$ are denoted as  $P_k(\T)$,  their vector- or matrix-valued 
versions  as  $P_k(\T;\R^2)$ or $P_k(\T;\R^{2\times 2})$ etc. 
The  subordinated Arnold-Winther finite element space $AW_k(\T)$ of index $k\in \mathbb{N}$
\cite{AW02} reads
\[
 AW_k(\T):=\big\{\tau\in P_{k+2}(\T;\mathbb S)\cap H(\ddiv,\Omega;\mathbb{S})\,
 \big|\,\ddiv\tau\in P_k(\T;\mathbb{R}^2)\big\}.
\]
The Neumann boundary conditions are essential conditions and are traditionally implemented by some approximation
$g_h$ to $g$ in the normal trace space 
\[
G(\T):=\{  (\tau_h \nu) |_{\Gamma_N} \in L^2(\Gamma_N;\R^2) \,\big|\,
 \tau_h\in  AW_k(\T) \}
\]
(recall that $\nu$ is the exterior unit normal along the boundary). Given any $g_h\in G(\T)$, 
the discrete stress approximations are sought in the non-void affine subspace 
\[
 \Sigma(g_h,\T):=\Sigma(g)\cap AW_k(\T)
\]
of  $AW_k(\T)$ with test functions in the  linear subspace  $\Sigma(0,\T):=\Sigma_0\cap AW_k(\T)$. 
Then there exists a unique 
discrete solution  $\sigma_h\in\Sigma(g_h,\T)$ and $ u_h\in V_h:= P_k(\T;\R^2)$ to
\begin{align}\label{eq:discpp}
\begin{split}
\int_\Omega\sigma_h :\C^{-1}\tau_h\,dx+\int_\Omega u_h\cdot\ddiv\tau_h\,dx
& = \int_{\Gamma_D} u_D \cdot(\tau_h\nu)\, ds 
\quad \textrm{for all } \tau_h\in\Sigma(0,\T),\\
\int_\Omega v_{h}\cdot \ddiv\sigma_h\, dx & = 
\int_\Omega f\cdot v_{h}\, dx\quad\textrm{ for all }  v_{h}\in V_h.
\end{split}
\end{align}
The explicit design of a Fortin projection  leads in \cite{AW02} to quasi-optimal a~priori error estimates
for an exact solution $(\sigma, u)\in(\Sigma(g)\cap H^{k+2}(\Omega;\mathbb S))
\times  H^{k+2}(\Omega)$ to \eqref{eq:lame}
and the approximate solution $(\sigma_h, u_h)$ to 
\eqref{eq:discpp}, namely (with the maximal mesh-size $h$)
\begin{align*}
  \|\sigma-\sigma_h\|_{L^2(\Omega)}&\lesssim h^{m}\|\sigma\|_{H^m(\Omega)}\quad\mbox{for } 1\leq m \leq k+2,\\
  \| u- u_h\|_{L^2(\Omega)}&\lesssim h^{m}\|u\|_{H^{m+1}(\Omega)}\quad\mbox{for } 1\leq m \leq k+1. 
\end{align*}

Another stable pair of different and mesh-depending norms in  \cite{CGS} 
 implies the $L^2$ best approximation of  the stress error $\sigma-\sigma_h$ 
up to a generic multiplicative constant and data oscillations on $f$
under some extra condition (N) on the Neumann data approximation $g_h$ implied by
the first and zero moment orthogonality assumption $g-g_h\perp  P_1( \mathcal{E}(\Gamma_N);\R^2)$ 
($\perp$ indicates orthogonality in $L^2(\Gamma_N)$) met in all the numerical examples of this paper.

For simple benchmark examples with piecewise polynomial data $f$ and $g$, there is even 
a superconvergence phenomenon   visible in numerical examples. The arguments of this paper
allow a proof of  fourth-order convergence of the $L^2$ stress error $\|\sigma-\sigma_h\|=o( h^4 )$
in the lowest-order Arnold-Winther method with $k=1$
for a smooth stress $\sigma\in H^4(\Omega;\mathbb{S})$  with  $f=f_h\in P_1(\T;\R^2)$ and $g=g_h\in G(\T)$.
(In fact, once the data are not piecewise affine, the arising oscillation terms are only of at most third order and 
the aforementioned convergence estimates are sharp.)  

This is stated as  Theorem~\ref{supperconvergence} in the appendix, because the a~priori error analysis
lies  outside of the  main focus of this work. It is surprising though that adaptive mesh-refining suggested
with this paper recovers this higher convergence rate even for the inconsistent Neumann data in the 
Cook membrane  benchmark example below.

%----------------------------------------------------------------------------------------------------
\subsection{Explicit residual-based  a~posteriori error estimator}
%----------------------------------------------------------------------------------------------------
The novel explicit residual-based error estimator for the discrete solution $(\sigma_h, u_h)$ to \eqref{eq:discpp}
depends only on the Green strain approximation $\C^{-1}\sigma_h$ and its piecewise derivatives and 
jumps across edges. 

Given any edge $E$ of length $h_E$, let $\nu_E$ denote the unit normal vector 
(chosen with a fixed orientation such that it points outside along the boundary $\partial\Omega$ 
of $\Omega$) and let $\tau_E$ denote its tangential unit vector; by convention 
$\tau_E =  (0,-1; 1,0) \nu_E$ with the indicated asymmetric $2\times 2$  matrix. The tangential 
derivative $\tau_E\cdot \nabla\bullet$ along an edge (or boundary)  is identified with the 
one-dimensional derivative $\partial \bullet/\partial s$ with respect to the arc-length 
parameter $s$.
The jump $[v]_E$ of any piecewise continuous scalar, vector, or matrix $v$ across an interior edge 
$E = \partial T_+\cap \partial T_-$ shared by the two triangles $T_+$ and $T_-$ such that
$\nu_E$ points outside $T_+$ along $E\subset \partial T_+$ reads
\[
[v]_E := (v|_{T_+})|_E -(v|_{T_-})|_E.
\]
The rotation acts on a vector field $\Phi$  (and row-wise on matrices) via 
$\rot \Phi  := \partial_1 \Phi_2 - \partial_2 \Phi_1$ and 
$\rot_{NC}$ denotes its piecewise application.

Under the present notation and the throughout abbreviation 
$\ee:=\C^{-1}\sigma_h$, the explicit residual-based   a~posteriori error estimator reads 
\begin{eqnarray}\label{eq:estimator}\nonumber
&\hspace{-20mm} 
\displaystyle \eta^2(\T,\sigma_h):=\sum_{T\in\T}h_T^4 \|\rot\rot\ee \|_{L^2(T)}^2 +\osc^2(f,\T)
+  \osc^2(g-g_h,\E(\Gamma_N)) \\ &
\hspace{-5mm}+\hspace{-2mm}\displaystyle
\sum_{E\in\E(\Omega)}\hspace{-2mm} \left(\hspace{-1mm} h_E \|\tau_E \cdot [\ee]_E \tau_E \|_{L^2(E)}^2
+ h_E^3 \| \tau_E\cdot (   [\rot_{NC}\ee]_E - \frac{ \partial  [\ee]_E \nu_E }{\partial s })\|_{L^2(E)}^2
\hspace{-1mm}\right)
\\ &\hspace{-1mm} +\hspace{-3mm}\displaystyle \sum_{E\in\E(\Gamma_D)} 
\hspace{-2mm} \left(\hspace{-1mm}h_E\|\tau_E \cdot (\ee \tau_E - \frac{ \partial u_D}{\partial s}) \|_{L^2(E)}^2 + 
h_E^3 \|\tau_E\cdot \rot \ee  - \nu_E\cdot (\frac{ \partial \ee \tau_E }{\partial s }
         +\frac{ \partial^2 u_D}{ \partial s^2}) \|_{L^2(E)}^2\hspace{-1mm}\right) \nonumber
 \end{eqnarray}
for the oscillations $\osc(f,\T)$ of the volume force and the oscillations $\osc(g-g_h,\E(\Gamma_N))$
of the traction boundary condition, defined below. 

\begin{theorem}[reliability]\label{thmREL}
There exists a mesh-size and $\lambda$ independent constant $C_{\textrm{rel}} $  
(which may depend on $\mu$ and on the shape-regularity of the triangulation $\T$ 
through a global lower bound of the minimal angle therein) 
such that the exact (resp. discrete) stress 
$\sigma$  from \eqref{eq:lame} (resp. $\sigma_h$ from  \eqref{eq:discpp} with 
$g-g_h\perp P_0(\mathcal{E}(\Gamma_N);\R^2)$
and the error estimator \eqref{eq:estimator} satisfy
\begin{align*}
  \| \sigma - \sigma_h \|_{L^2(\Omega)} \leq C_{\textrm{rel}}  \eta(\T,\sigma_h).
\end{align*}
\end{theorem}

The a~posteriori error estimator $\eta (\T,\sigma_h)$  already involves two data oscillation 
terms $\osc(f,\T)$ and  $\osc(g-g_h,\E(\Gamma_N))$ defined as the square roots of the respective terms in  
\begin{align*}
&\osc^2(f,\T):= \sum_{T\in\T} h_T^2\|f-f_h\|_{L^2(T)}^2 \text{  for the }L^2\text{ projection }f_h \text{ of } f
\text{ onto } P_k(\T;\R^2); \\
&\osc^2(g-g_h,\E(\Gamma_N)):= \sum_{E\in\E(\Gamma_N)} h_E \|g-g_h\|_{L^2(E)}^2.
\end{align*}

For any edge $E$ and a degree $m\ge k+2$, let $\Pi_{m,E}:L^2(E)\to P_{m}(E)$ denote the $L^2$ projection 
onto polynomials of degree at most $m$. For any $E\in\E(\Gamma_D)$ define the 
two Dirichlet data oscillation terms 
\begin{eqnarray}
\label{eqdefoscIondirichlet}
\osc^2_I(u_D,E)&:=& h_E\| (1-\Pi_{m,E}) \partial (u_D\cdot\tau_E)/\partial s \|_{L^2(E)}^2, \\
\label{eqdefoscIIondirichlet}
\osc^2_{II}(u_D,E)&:= &
h_E^3\| (1-\Pi_{m,E}) \partial^2 ( u_D\cdot\nu_E )/ \partial s^2\|_{L^2(E)}^2.
\end{eqnarray}
Their sum defines the overall Dirichlet data approximation $\osc(u_D,\E(\Gamma_D))$ as the square root of
\[
\osc^2(u_D,\E(\Gamma_D)):=\sum_{E\in \E(\Gamma_D)}
\left(\osc^2_I(u_D,E)+\osc^2_{II}(u_D,E)\right).
\]

The analysis of Section~\ref{sec:efficiency} is local and states for each of the five local residuals an upper bound
related to the error in a neighborhood. The global efficiency is displayed as follows. 
 
\begin{theorem}[efficiency]\label{thmEFF}
There exists a mesh-size and $\lambda,\mu $ independent constant 
$C_{\textrm{eff}} $ (which may depend on the shape-regularity of the triangulation $\T$ 
through a global lower bound of the minimal angle therein) 
such that the exact (resp. discrete) stress 
$\sigma$  from \eqref{eq:lame} (resp. $\sigma_h$ from  \eqref{eq:discpp} with 
$g-g_h\perp P_0(\mathcal{E}(\Gamma_N);\R^2)$
and the error estimator \eqref{eq:estimator} satisfy
\begin{align*}
C_{\textrm{eff}}^{-1}  \eta(\T,\sigma_h) \leq   \| \sigma - \sigma_h \|_{L^2(\Omega)}
+ \osc(f,\T)+\osc(g-g_h,\E(\Gamma_N))+\osc(u_D,\E(\Gamma_D)).
\end{align*}
\end{theorem}

%----------------------------------------------------------------------------------------------------
\subsection{Outline of the paper}
%----------------------------------------------------------------------------------------------------
The remaining parts of this paper provide a mathematical proof of 
Theorem~\ref{thmREL} and  Theorem~\ref{thmEFF}
and numerical evidence  in computational experiments 
on the novel a~posteriori error estimation and its robustness as well as 
on associated mesh-refining algorithms. 

The proof of the reliability of Theorem~\ref{thmREL} in Section~\ref{sec:reliability} adopts  
arguments  of  \cite{CD_1998,CG_2016} and carries out  two integration by parts on each triangle plus 
one-dimensional integration by parts along all edges. 
The resulting terms are in fact locally efficient in Section~\ref{sec:efficiency} 
with  little generalizations of the bubble-function methodology due to Verf\"urth \cite{Verfuerth1996}. 
The five lemmas of 
Section~\ref{sec:efficiency} give  slightly sharper results and in total imply Theorem~\ref{thmEFF}.

The point in Theorem~\ref{thmREL} and \ref{thmEFF} is that the universal constants
$C_{\textrm{rel}}$ and $C_{\textrm{eff}}$ may depend on the Lam\'e
parameter $\mu$ but are independent of the critical
Lam\'e parameter $\lambda$ as supported by  the benchmark examples of the concluding 
Section~\ref{sec:examples}. 
Adaptive mesh-refining proves to be highly effective with the novel a~posteriori error 
estimator even for incompatible Neumann data.
Four benchmark examples with the Poisson ratio $\nu=0.3$  or $0.4999$ provide numerical evidence 
of the robustness of the reliabile and efficient a~posteriori error estimation and for the fourth-order 
convergence of
 Theorem~\ref{supperconvergence}. 

\subsection{Comments on general notation}
Standard notation on Lebesgue and Sobolev
spaces and norms is adopted throughout this paper and, for brevity,
$\|\cdot\|:=\|\cdot\|_{L^2(\Omega)}$ denotes the $L^2$ norm.
The piecewise action of a differential operator is denoted with a subindex ${NC}$, e.g., 
$\nabla_{NC}$ denotes the piecewise gradient
$(\nabla_{NC} \bullet)|_T  := \nabla (\bullet |_T)$ for all $T\in\T$.
Sobolev functions are usually defined on open sets and the notation $W^{m,p}(T)$ 
(resp. $W^{m,p}(\T)$) 
substitutes  $W^{m,p}(\operatorname{int}(T))$ for a (compact) triangle $T$ and its interior
$\operatorname{int}(T)$ (resp. $W^{m,p}(\operatorname{int}(\T))$)
and their vector and matrix versions.

For a differentiable function  $\phi$, $\Curl \phi := (-\partial_2 \phi , \partial_1 \phi)$ is the rotated
gradient;  
for a two-dimensional vector field $\Phi$,  $\Curl \Phi$ is the $2\times 2$ matrix-valued 
rotated gradient
\[
\Curl \Phi := (-\partial_2 \Phi_1 , \partial_1 \Phi_1 ;
-\partial_2 \Phi_2 , \partial_1 \Phi_2 )
=D\Phi(0,1; -1,0).
\]
(The signs are not uniquely determined in the literature and some care is required.) 

The colon denotes the  scalar product 
$A:B:=\sum_{\alpha,\beta=1,2}  A_{\alpha,\beta} B_{\alpha,\beta}$
of $2\times 2$ matrices $A,B$. 
The inequality $A\lesssim B$ between two terms $A$ and $B$ abbreviates  $A\le C\, B$
with some multiplicative generic constant $C$, which is independent of the mesh-size and independent of the 
one Lam\'e parameter $\lambda\ge 0$ but may depend on the other $\mu>0$ and may depend on the 
shape-regularity of the underlying triangulation $\T$.

%----------------------------------------------------------------------------------------------------
\section{Proof of reliability}
\label{sec:reliability}
%----------------------------------------------------------------------------------------------------
This section is devoted to the proof of Theorem~\ref{thmREL}
based on a Helmholtz decomposition on \cite{CD_1998} 
with two parts as decomposed in Theorem~\ref{thm:Helmholtz} below. The critical part is the
$L^2$ product of $\C^{-1} (\sigma-\sigma_h)$ times the $\Curl$ of an unknown function $\Curl \beta$. 
The observation from \cite{CG_2016} is that one can find an Argyris finite element approximation $\beta_h$
to $\beta\in H^2(\Omega)$ such that the continuous function $\phi:=\beta-\beta_h\in H^2(\Omega)$ 
vanishes at all vertices $\N$ of the triangulation. Two integration
by parts  on each triangle plus one-dimensional integration by parts along the edges $\E$ of the triangulation eventually
lead to a key identity.  

\begin{lemma}[representation formula]\label{lemma:representation}
Any function $\ee\in  H^2(\T;\mathbb{S})$ 
(i.e. $\ee$ is piecewise in $H^2$ with values in $\mathbb{S}$) and any $\phi\in H^2(\Omega)$ 
with $\phi(z)=0$ at all vertices $z\in\N$ in the regular triangulation $\T$ satisfy
\begin{eqnarray*}
& \hspace{-20mm} (\ee, \Curl^2 \phi)_{L^2(\Omega)} =(\rot_{NC}\rot_{NC}\ee,\phi)_{L^2(\Omega)}
  \\ & \displaystyle  +\sum_{E\in\E(\Omega)} \left(    (\tau_E\cdot[\ee]_E\tau_E,\partial_{\nu_E}\phi)_{L^2(E)}
  -   ([\rot_{NC}\ee]_E - \frac{ \partial [\ee]_E\nu_E}{\partial s} ,\phi \, \tau_E )_{L^2(E)} \right) 
  \\& \displaystyle  +\sum_{E\in\E(\partial\Omega)} \left(  (\tau_E\cdot\ee\tau_E,\partial_{\nu_E}\phi )_{L^2(E)}
  - (\rot\ee -\frac{ \partial \ee\nu_E}{\partial s} , \phi \, \tau_E)_{L^2(E)}\right) .
\end{eqnarray*}
\end{lemma}

The subsequent integration by parts formula is utilized frequently throughout this paper 
for $\phi\in H^1(\Omega;\R^2)$ and $\Psi\in H^1(\Omega;\R^{2\times 2})$ 
\[
\int_\Omega  \Psi:\Curl \phi \, dx+  \int_\Omega  \phi \cdot \rot \Psi\, dx 
=  \int_{\partial\Omega}  \phi  \cdot  \Psi\tau_E\, ds.
\]
Any differentiable (scalar) function $\varphi$, satisfies  the elementary relations 
\[
\tau_E\cdot \Curl \varphi =  \partial \varphi/\partial \nu_E
\quad\text{and}\quad \nu_E\cdot \Curl \varphi = - \partial \varphi/\partial s=- \partial \varphi/\partial\tau_E 
\quad\text{on }E\in\E.
\]

\begin{proof}
Integrate by parts twice on each 
triangle and rearrange  the remaining  boundary terms to deduce
(with the abbreviation $\rot_{NC}\rot_{NC}\equiv\rot_{NC}^2$)
\begin{eqnarray*}
&&\hspace{-10mm} (\ee,\Curl^2\phi)_{L^2(\Omega)}  = (\rot_{NC}^2\ee,\phi)_{L^2(\Omega)} \\
&&+\sum_{E\in\E(\Omega)} \left(    ([\ee]_E\tau_E,\Curl\phi)_{L^2(E)}
-([\rot_{NC}\ee]_E\cdot\tau_E,\phi )_{L^2(E)} \right) \\ 
&&+\sum_{E\in\E(\partial \Omega)} \left(  (\ee\tau_E,\Curl\phi)_{L^2(E)}
-(\rot\ee\cdot\tau_E,\phi )_{L^2(E)}  \right) .
\end{eqnarray*}
The term $  ([\ee]_E\tau_E,\Curl\phi)_{L^2(E)} $ in the above sum is split into  orthogonal components 
\[
\Curl\phi=  (\tau_E\cdot \Curl\phi)\tau_E+ (\nu_E\cdot \Curl\phi)\nu_E=
(\tau_E\cdot \Curl\phi)\tau_E-  (\partial \phi/\partial s )\nu_E\quad\text{on }E\in\E.
\]
Since $\phi$ vanishes at the vertices,
an integration by parts along each  interior edge $E$ 
for the last term shows 
$ ([\ee]_E\tau_E, (\partial \phi/\partial s )\nu_E)_{L^2(E)} = -  (\partial [\ee]_E\tau_E/\partial s  , \phi\nu_E)_{L^2(E)}$. 
This proves 
\[
([\ee]_E\tau_E,\Curl\phi)_{L^2(E)}  =  (\tau_E\cdot [\ee]_E\tau_E,\partial_{\nu_E}\phi)_{L^2(E)}  +
 (\frac{\partial \nu_E\cdot[\ee]_E\tau_E}{\partial s},\phi)_{L^2(E)}.
\]
The same formula holds for any boundary edge $E$ when $[\ee]_E$ is replaced by $\ee$. The combination of the
latter identities with the first displayed formula of this proof verifies the asserted representation formula.
\end{proof}

The contribution of $\epsilon(u)=\C^{-1} \sigma$ times the $\Curl^2\phi\in L^2(\Omega;\mathbb{S})$ 
exclusively leads  to boundary terms. Throughout this paper, suppose that the Dirichlet data $u_D$ satisfies 
$u_D\in C(\Gamma_D)\cap H^2(\E(\Gamma_D)$ 
in the sense that  $u_D$  is globally continuous 
with  $u_D|_E\in H^2(E;\R^2)$ for all $E\in \E(\Gamma_D)$. 

\begin{lemma}[boundary terms]\label{lemma:bd}
Any Sobolev function  $v\in H^1(\Omega;\R^2) $  with boundary values 
$u_D\in C(\Gamma_D)\cap H^2(\E(\Gamma_D)$ on $\Gamma_D$
and any $\phi\in H^2(\Omega)$ with $\phi=\partial \phi/\partial\nu=0$ along $\Gamma_N$ 
with $\phi(z)=0$ for any vertex $z$ of $\Gamma_D$ in its relative interior satisfy 
\[
 (\varepsilon (v) , \Curl^2 \phi)_{L^2(\Omega)}=\sum_{E\in\E(\Gamma_D)} 
   \left(  (\frac{ \partial u_D}{\partial s  } ,  \frac{ \partial\phi}{\partial \nu_E}\, \tau_E )_{L^2(E)}
  + (\frac{\partial^2   u_D}{\partial s^2 }  ,\phi\, \nu_E )_{L^2(E)}\right).
\]
\end{lemma}

\begin{proof}
A density argument shows that it suffices to prove this identity for smooth functions $v$ and $\phi$, when 
integration by parts arguments show that the left-hand side is equal to  
\[
\int_{\partial\Omega} \Curl \phi \cdot \frac{\partial v }{\partial s}ds=
\sum_{E\in \E(\partial\Omega)}
\int_{E} \left( \frac{\partial\phi}{\partial \nu_E} \frac{\partial (v\cdot\tau_E)}{\partial s}
+ \phi \frac{\partial^2 (v\cdot \nu_E) }{\partial s^2}\right)ds .
\]
The equality  follows from an orthogonal split 
$\Curl\phi=  (\tau\cdot \Curl\phi)\tau+ (\nu\cdot \Curl\phi)\nu$ into the normal and tangential
directions of $\nu$ and $\tau$ along the boundary $\partial\Omega$ followed by an integration by parts
along $\partial\Omega$ with $\phi(z)=0$ for vertices $z$ in $\Gamma_D$ with a jump of the normal
unit vector.  The substitution of the boundary conditions concludes the proof.
\end{proof}

The consequence of the previous two lemmas is a representation formula 
for the error times a typical function $\Curl^2 \phi$. To understand why the contributions on the Neumann
boundary of $\phi$ and $\nabla \phi$  disappear  along $\Gamma_N$, 
some details on the Helmholtz decomposition are recalled from the literature. For this, let 
$\Gamma_0,\dots, \Gamma_J$ denote the compact connectivity components of $\overline{\Gamma_N}$. 

\begin{theorem}[Helmholtz decomposition {\cite[Lemma 3.2]{CD_1998}}]
\label{thm:Helmholtz}
Given any $\sigma-\sigma_h\in L^2(\Omega;\mathbb{S})$, 
there exists $\alpha\in V$,  constant vectors  
$c_0,\dots,c_J\in\R^2$ with $c_0=0$ and  $\beta\in H^2(\Omega)$ 
with $\int_\Omega\beta\, dx = 0$ and  
$\Curl\beta = c_j $  on $\Gamma_j\subseteq\Gamma_N$ for all $j=0,\dots,J$  such that
	\begin{align}\label{e:helmholtz}
		\sigma-\sigma_h = \C\varepsilon(\alpha) + \Curl\Curl \beta. \qed
	\end{align}
\end{theorem}

The second ingredient is an approximation $\beta_h$ of $\beta$ 
from the Helmholtz decomposition in Theorem~\ref{thm:Helmholtz}
based on the Argyris finite element functions $ A(\T) \subset  C^1(\Omega)\cap P_5(\T)$ 
 \cite{BrennerScott2008,Braess2007,Cia78}. 
The local mesh-size $h_\T\in P_0(\T)$ in  the triangulation $\T$ is defined as its diameter $h_\T|_T:=h_T$
on each triangle $T\in\T$.

\begin{lemma}[quasiinterpolation]\label{lemma:quasi}
Given any $\beta$ as in Theorem~\ref{thm:Helmholtz}  there exists some $\beta_h\in A(\T)$
such that $\phi:=\beta-\beta_h \in H^2(\Omega) $  
vanishes at any vertex $z\in \N$ of the triangulation,  $\phi$ and its gradient $\nabla\phi$ 
vanish on $\Gamma_N$, and  the local approximation and stability property holds in the sense that 
\[
\| h_\T^{-2} \phi \| + \| h_\T^{-1}\Curl \phi \| +\| \Curl^2\phi \| \lesssim  \| \beta \|_{H^2(\Omega)}.
\]
\end{lemma}

\begin{proof}
This has been (partly) utilized in \cite{CG_2016}  and also follows from \cite{GS_2002}. 
\end{proof}

The combination of all aforementioned arguments leads to the following estimate 
as an answer to the question of Subsection~\ref{subsec:1.1} in terms of directional derivatives
of $\ee:= \C^{-1}\sigma_h$. Recall the definition of  $\eta(\T,\sigma_h)$ from \eqref{eq:estimator}.

\begin{theorem}[key result]\label{thm:keyresult}
Let $\sigma\in H(\ddiv,\Omega;\mathbb{S})$ solve \eqref{eq:lame} and let $\sigma_h \in AW_k(\T)$
solve \eqref{eq:discpp}. 
Given $\beta$ from Theorem~\ref{thm:Helmholtz} and its quasiinterpolation 
$\beta_h$ from Lemma~\ref{lemma:quasi}, 
the difference $\phi:=\beta-\beta_h$ satisfies
\[
(\C^{-1}(\sigma-\sigma_h), \Curl^2\phi)_{L^2(\Omega)} \lesssim  | \beta |_{H^2(\Omega)} \eta(\T,\sigma_h).
\]
\end{theorem}

\begin{proof}
Lemma~\ref{lemma:representation} and Lemma~\ref{lemma:bd} lead to a formula for 
$(\ee, \Curl^2\phi)_{L^2(\Omega)}$, $\ee:=  \C^{-1}\sigma_h$, in which all 
the contributions for $E\in\E(\Gamma_N)$
 with $\phi$ and $\nabla\phi$ vanish along $\Gamma_N$. The remaining formula reads 
 \begin{equation*}
 \begin{aligned}
 &(\C^{-1}(\sigma-\sigma_h),\Curl^2\phi)_{L^2(\Omega)}
 = - (\rot_{NC}^2\ee,\phi)_{L^2(\Omega)}
   \\ & \displaystyle  -\sum_{E\in\E(\Omega)} \left(    
   (\tau_E\cdot[\ee]_E\tau_E,\frac{\partial \phi  }{\partial \nu_E})_{L^2(E)}
  -   ([\rot_{NC}\ee]_E - \frac{ \partial [\ee]_E\nu_E}{\partial s} ,\phi \, \tau_E )_{L^2(E)} \right) 
  \\& +\sum_{E\in\E(\Gamma_D)} \left( 
   ( \frac{\partial u_D}{\partial s} - \ee\tau_E, \tau_E\,  \frac{\partial\phi}{\nu_E} )_{L^2(E)} \right.
  \\&\qquad \quad\left. + (     \tau_E\cdot (  \rot_{NC}\ee -\frac{\partial( \ee\nu_E)}{\partial s} )
     +\frac {\partial^2  u_D\cdot\nu_E}{\partial s^2 } ,\phi )_{L^2(E)}\right).
 \end{aligned}
\end{equation*}
The  proof  concludes with 
Cauchy-Schwarz inequalities, trace inequalities,  and the  approximation estimates 
of Lemma~\ref{lemma:quasi}. The remaining details are nowadays standard arguments in the a~posteriori 
error analysis of nonconforming and mixed finite element methods and hence are omitted.
\end{proof}

\bigskip

Before the proof of Theorem~\ref{thmREL} concludes this section, three remarks and one lemma are in order.

\begin{remark}[nonconstant coefficients]\label{rematkcc1}
The main parts of the reliability analysis of this section hold for rather general material tensors 
$\C$ as long as $\ee:= \C^{-1}\sigma_h$ allows for the existence of the traces and the derivatives in the 
error estimator  \eqref{eq:estimator}  in the respective $L^2$ spaces. For instance, if $\lambda$ and $\mu$ 
are piecewise smooth with respect to the underlying triangulation $\T$. 
\end{remark}

\begin{remark}[constant coefficients]\label{remarkonjumptermvanishes}
The overall assumption of constant Lam\'e parameters $\lambda$ and $\mu$ allows a simplification in  the 
error estimator  \eqref{eq:estimator}. It suffices to have $\mu$ globally continuous and $\mu$ and $\lambda$ 
piecewise smooth to guarantee 
\[
\frac{ \partial [\ee]_E\nu_E}{\partial s} \cdot \tau_E =0\quad\text{along }E\in\E(\Omega).
\]
(The proof utilizes the structure of $\C$ and $\C^{-1}$
with $\C^{-1} E = \frac{1}{2\mu} (E - \frac{\lambda}{2(\lambda+\mu)} \tr(E) 1_{2\times 2})$
for any $E\in\mathbb{S}$
 as a linear combination of the identity and some 
scalar multiple of the $2\times 2$ unit matrix. The terms with the identity lead to $1/(2\mu)$ times the jump
$[\sigma_h]_E\nu_E=0$ of the $H(\ddiv)$ conforming stress approximations. The jump terms with the 
unit matrix  (even with jumps of $\lambda$)  are multiplied with the orthogonal unit vectors
$\nu_E$ and $\tau_E$ and so vanish as well.) 
\end{remark}

\begin{remark}[Related work] 
Although the work  \cite{HHX2011} concerns a different problem (bending of a plate
of fourth order) with a different discretisation (even nonconforming in $H(\ddiv)$),
some technical parts of that paper are related to those of this by a rotation of the  underlying coordinate system 
and the substitution  of $\ddiv\ddiv$ 
by  $\rot\rot$ etc.  Another Helmholtz decomposition also allows for a discrete version and thereby
enables a proof of optimal 
convergence of an adaptive algorithm with arguments from \cite{CFPP,CRabus}. 
\end{remark}

A technical detail related to the  robustness in $\lambda\to\infty$ is a well known lemma that controls
the trace of a matrix $E\in\R^{2\times 2}$ by its deviatoric part $\operatorname{dev} E:= E-\tr(E)/2\, 1_{2\times 2}$ 
and its divergence measured in the dual $V^*\subset H^{-1}(\Omega;\R^2)$ of $V$, namely
\[
\|  \ddiv \tau \|_{-1} := \sup_{\substack{v\in V \\   | v |_{H^1(\Omega)} =1}}{\int_\Omega \tau: Dv\, dx}
\quad\text{for all  } 
\tau\in L^2(\Omega;\R^{2\times 2}).
\]

\begin{lemma}[tr-dev-div]\label{l:trdevdivCCDolzmann}
Let $\Sigma_0$ be a closed subspace of
$H(\ddiv,\Omega;\mathbb R^{2\times 2})$, which does not contain the
constant tensor $1_{2\times 2}$. Then any $\tau\in \Sigma_0$ satisfies
\begin{align*}
\|\operatorname{tr}(\tau)\|_{L^2(\Omega)}
 \lesssim \|\operatorname{dev}\tau\|_{L^2(\Omega)}   + \|\ddiv\tau\|_{-1}.
\end{align*}
\end{lemma}

\begin{proof}
There are several variants of the tr-dev-div lemma known in the literature 
[BF91, Proposition 7 in Section IV.3]. The version in \cite[Theorem~4.1]{CD_1998}  explicitly displays 
a version with  $\|\ddiv\tau\|$ replacing $\|\ddiv\tau\|_{-1}$. Since its proof is immediately adopted 
to prove the asserted version, further details are omitted.
\end{proof}

\bigskip

The remaining part of this section outlines why Theorem~\ref{thmREL} follows from 
Theorem~\ref {thm:keyresult} with the arguments from  \cite{CD_1998,CG_2016}.   
The energy norms of  $v\in V$ and $\tau\in H(\ddiv,\Omega;\mathbb{S}) $ read
\[
\anorm{v}^2:=\int_\Omega \varepsilon(v):\C\varepsilon(v)\, dx \quad\text{and}\quad
\|\tau\|_{\C^{-1}}^2:=\int_\Omega\tau : \C^{-1}\tau\, dx.
\]
The remaining residual 
\begin{align*}
		\Res(v) := \int_\Omega f\cdot v\, dx + \int_{\Gamma_N} g\cdot v\, ds 
		- \int_\Omega \sigma_h:\varepsilon(v)\, dx\quad\text{for all }v\in V 
\end{align*}
with its dual norm
\begin{align*}
\anorm{\Res}_* := \sup_{\substack{v\in V \\ \anorm{v}=1}}{\Res(v)}
\end{align*}
allow the following error decomposition \cite[Theorem 3.1]{CG_2016},
with the weighted $L^2$ 
norms $\|\bullet\|_{\C^{-1}}:=\| \C^{-1/2} \bullet \|_{L^2(\Omega)}$
and  $\|\bullet \|_{\C}:=\| \C^{1/2} \bullet \|_{L^2(\Omega)}$,
\begin{align} \label{e:errorsplit}
\|\sigma-\sigma_h\|_{\C^{-1}}^2 
= \anorm{\Res}_{*}^2 + 
\min_{w\in u_D+V}\|\C^{-1}\sigma_h-\varepsilon(w)\|_\C^2.
\end{align}
In fact, it is shown in the proof of \cite[Theorem 3.1]{CG_2016} that $\alpha\in V$ and $\beta\in H^2(\Omega)$ 
from the Helmholtz decomposition of the error $\sigma-\sigma_h$ in 
Theorem~\ref{thm:Helmholtz} are orthogonal with respect to the $L^2$ scalar product weighted with $\C$ and that
the last term in \eqref{e:errorsplit} equals $\|\Curl^2\beta\|_{\C^{-1}}^2=
(\C^{-1}(\sigma-\sigma_h), \Curl^2\phi)_{L^2(\Omega)}$ and hence is
controlled in the key estimate of Theorem~\ref{thm:keyresult}. 

Lemma~\ref{l:trdevdivCCDolzmann} applies to $\Sigma_0$ as the subspace of all $\tau\in H(\ddiv,\Omega;\mathbb{S})$
with homogeneous Neumann data $\tau\nu=0$ along $\Gamma_N$.
Since   $\tau=\Curl^2\beta$ is divergence free and since $\tau\nu=-\partial \Curl\beta /\partial s$ along 
$\Gamma_N$
(owing to the aforementioned elementary relations and the convention that the first $\Curl$ acts row-wise on $\Curl\beta$), 
where $\Curl\beta$  in Theorem~\ref{thm:Helmholtz} is piecewise constant,
it follows that $\tau\in\Sigma_0$. On the other hand $1_{2\times 2}\notin \Sigma_0$ because $\Gamma_N\ne \emptyset$.
Consequently,  Lemma~\ref{l:trdevdivCCDolzmann} implies 
$\| \Curl^2 \beta \|\lesssim  \| \operatorname{dev} \Curl^2 \beta \|$. 
This and elementary calculations with $\C^{-1}$
lead to
\[
| \beta |_{H^2(\Omega)} = \| \Curl^2 \beta \|\lesssim  \| \operatorname{dev} \Curl^2 \beta \|
 \lesssim  \| \Curl^2 \beta\|_{\C^{-1}}.
\]
The combination with the  estimate resulting from Theorem~\ref{thm:keyresult} proves 
\[
\min_{w\in u_D+V}\|\C^{-1}\sigma_h-\varepsilon(w)\|_\C^2 
= \|  \Curl^2\beta\|_{\C^{-1}}^2
\lesssim \|  \Curl^2\beta\|_{\C^{-1}}\,   \eta(\T,\sigma_h).
\]
This implies  $\|  \Curl^2\beta\|_{\C^{-1}}\lesssim  \eta(\T,\sigma_h) $ and leads in 
\eqref{e:errorsplit} to 
\begin{equation}\label{eq:ccreliabilirtyyproof1}
\|\sigma-\sigma_h\|_{\C^{-1}} \lesssim \anorm{\Res}_{*} +  \eta(\T,\sigma_h).
\end{equation}
The remaining term is the estimate of the dual norm $\anorm{\Res}_{*}$ of the residual 
which is done, e.g., in 
{\cite[Lemma 3.3]{CG_2016}}  (under the assumption $g-g_h\perp P_0(\E(\Gamma_N))$)
\[
\anorm{\Res}_{*} \lesssim  \osc(f,\T)+ \osc(g-g_h,\E(\Gamma_N))\le  \eta(\T,\sigma_h)
\]
This and  \eqref{eq:ccreliabilirtyyproof1}
imply
\[
\|\operatorname{dev}(\sigma-\sigma_h)\|\lesssim
\|\sigma-\sigma_h\|_{\C^{-1}} \lesssim  \eta(\T,\sigma_h).
\] 
For any test function $v\in V$ with $ | v |_{H^1(\Omega)} =1$, 
$\int_\Omega (\sigma-\sigma_h) : Dv\, dx= \Res(v)$  and so 
\[
\|\operatorname{div}(\sigma-\sigma_h)\|_{-1} = 
 \sup_{\substack{v\in V \\   | v |_{H^1(\Omega)} =1}}{Res(v)}
 \le \sup_{\substack{v\in V \\   \| \epsilon(v) \|  =1}}{Res(v)} \le 2\mu\, \anorm{\Res}_{*}
  \lesssim  \eta(\T,\sigma_h).
\]
(In the second last step one utilizes that $ 2\mu\, E:E \le E:\C E$ for all $E\in\mathbb{S}$.)  
The combination of Lemma~\ref{l:trdevdivCCDolzmann} for $\tau=\sigma-\sigma_h$ 
with the previous displayed estimates concludes the proof of $\|\sigma-\sigma_h\|\lesssim \eta(\T,\sigma_h)$.
There exist several  appropriate choices of $\Sigma_0\subset H(\ddiv,\Omega;\mathbb{S})$ in this last step. 
Recall that $\Gamma_N$ is the union of connectivity components and so  pick one edge $E_0$
in this polygon and consider
 $\Sigma_0:=\{  \tau\in  H(\ddiv,\Omega;\mathbb{S}):   \int_{E_0} \tau\nu\, ds =0\}$ 
with $1_{2\times 2}\notin \Sigma_0$.
This choice of $E_0$  and so  $\Sigma_0$  depend only on $\Gamma_N$ (independent of $\T$). 
Since $g-g_h=(\sigma-\sigma_h)\nu$ along $E_0$ has 
(piecewise on $\E(E_0)$, whence in total) an integral mean zero, 
Lemma~\ref{l:trdevdivCCDolzmann} indeed applies 
to $\tau=\sigma-\sigma_h\in\Sigma_0$.
\qed

%----------------------------------------------------------------------------------------------------
\section{Local efficiency analysis}
\label{sec:efficiency}
%----------------------------------------------------------------------------------------------------
The local efficiency follows with the bubble-function technique
for $C^1$ finite elements \cite[Sec 3.7]{Verfuerth1996}.
This section focusses  on a constant $\C$ for linear isotropic elasticity with 
constant Lam\'e parameters $\lambda$ and $\mu$ such that
$\ee:=\C^{-1}\sigma_h\in P_{k+2}(\T)$  for some $\sigma_h\in AW_k(\T)$
is a polynomial of degree at most $k+2$. Appart from this, the Lam\'e parameters 
do not further arise in this section.

The moderate point of departure is the volume term for each triangle $T\in\T$
with barycentric coordinates $\lambda_1,\lambda_2,\lambda_3\in P_1(T)$ 
and their product, the cubic volume bubble function,  
$b_T:=27\,\lambda_1\lambda_2\lambda_3 \in W^{1,\infty}_0(T)$ plus its
square $b_T^2\in  W^{2,\infty}_0(T)$ with $0\le b_T^2\le 1$, 
$ \|b_T\|_{L^2(T)} \lesssim 1$, and $|b_T|_{H^2(T)} \lesssim h_T^{-2}$ etc.

\begin{lemma}[efficiency of volume residual]\label{lemma:eff:1}
Any $v\in H^1(T;\R^2)$, $T\in\T$, satisfies
\begin{align*}
h_T^2 \| \rot\rot \ee \|_{L^2(T)} \lesssim \| \ee - \varepsilon(v) \|_{L^2(T)}.
\end{align*}
\end{lemma}

\begin{proof}
An inverse estimate for the polynomial $\rot\rot \ee\equiv \rot^2\ee$  implies the estimate 
\[
  \| \rot^2 \ee \|_{L^2(T)}^2   \lesssim \| b_T^{1/2} \rot^2 \ee \|^2_{L^2(T)}
  = (\rot^2\ee , b_T \rot^2 \ee )_{L^2(T)}.
\]
Lemma~\ref{lemma:representation} with $\phi=b_T \rot^2 \ee$
and $( \varepsilon(v)  , \Curl^2\phi )_{L^2(T)} = 0$ lead to
\begin{align*}
\begin{split}
\| b_T^{1/2} \rot^2 \ee \|^2_{L^2(T)}
&= (   \ee - \varepsilon(v) , \Curl^2(b_T \rot^2 \ee ) )_{L^2(T)}\\
&\leq \| \ee - \varepsilon(v) \|_{L^2(T)} \|\Curl^2(b_T \rot^2 \ee )\|_{L^2(T)}.
\end{split}
\end{align*}
This and the inverse estimate 
$\|\Curl^2(b_T \rot^2 \ee )\|_{L^2(T)} \lesssim h_T^{-2} \|b_T \rot^2 \ee \|_{L^2(T)} $
imply
\[
  \| \rot^2 \ee \|_{L^2(T)}^2   \lesssim \| \ee - \varepsilon(v) \|_{L^2(T)}
  h_T^{-2} \|\rot^2 \ee \|_{L^2(T)}.
\]
This concludes the proof.
\end{proof}

The localization of the first edge residual  involves the piecewise quadratic edge-bubble function
$b_E$ with support $T_+\cup T_-$ for an interior edge $E=\partial T_+\cap \partial T_-$ shared by the two triangles
$T_+$ and $T_-$ with edge-patch  $\omega_E := \operatorname(T_+\cup T_-)$. 
With an appropriate scaling $b_E|_T=4\lambda_1\lambda_2$ 
for the two barycentric coordinates  $\lambda_1,\lambda_2$ on $T\in\{T_+,T_-\}$ associated with the 
two vertices of $E$. Then  $b_E\in W^{1,\infty}(\omega_E)$ and 
$b_E^2\in W^{1,\infty}(\omega_E)$ satisfy 
$0\le b_E^2\le b_E\le 1$ and  $|b_E|_{H^1(E)} \lesssim h_E^{-1}$ etc. 

\bigskip

The remaining technical detail is an extension of functions on the edge $E$ to $\omega_E$. Throughout this sections 
those functions are polynomials and given $\rho_E\in P_m(E)$, their coefficients (in some fixed basis) already 
define an algebraic object that  is a natural  extension $\rho\in P_m(\hat E)$ along the straight line 
$\hat E:=\operatorname{mid}(E)+\R\, \tau_E$ that
extends $E$ with midpoint  $\operatorname{mid}(E)$ and tangential unit vector $\tau_E$.  This and 
\[
P_E(\rho_E)(x):= \rho( \tau_E\cdot (x- \operatorname{mid}(E))) \quad\text{for all }x\in \R^2
\]
define a linear extension operator $P_E:P_m(E)\to C^\infty(\R^2)$ 
with $P_E(\rho_E)=\rho_E$ on $E$ for any $\rho_E\in P_m(E)$, which is constant in the normal direction,
$\nabla P_E(\rho_E)\cdot\nu_E\equiv 0$. This design is different from that in \cite{Verfuerth1996}.

\begin{lemma}[efficiency of first interior edge residual]\label{lemma:eff:2}
Any $v\in H^1(\omega_E;\R^2)$,  $E\in\E(\Omega)$, satisfies
\[
h_E^{1/2} \|\tau_E\cdot [\ee]_E\tau_E\|_{L^2(E)} \lesssim \| \ee - \varepsilon(v) \|_{L^2(\omega_E)}.
\]
\end{lemma}

\begin{proof}
Since $\tau_E\cdot [\ee]_E\tau_E\in P_{k+2}(E)$ is a polynomial, the above extension  
$P_E(\tau_E\cdot [\ee]_E\tau_E)$ and the function $b\in W^{2,\infty}_0(\omega_E)$ with  
\begin{equation}\label{testfunctioninlemma:eff:2}
b(x):= b_E^2(x)\, \nu_E\cdot (x- \operatorname{mid}(E))\quad\text{for all }x\in\R^2 
\end{equation}
define some  function $\phi:= b\, P_E(\tau_E\cdot [\ee]_E\tau_E)$.
Since $b=0$ and $\nabla b_E\cdot \nu_E= b_E^2$ along $E$, the test function $\phi\in H^2_0(\omega_E)\subset H_0^2(\Omega)$  leads in Lemma~\ref{lemma:representation} to
\[
  (\tau_E\cdot[\ee]_E\tau_E, \partial_{\nu_E} \phi)_{L^2(E)}=
  (\ee, \Curl^2 \phi)_{L^2(\omega_E)} - (\rot_{NC}^2\ee,\phi)_{L^2(\omega_E)}.
\]
Since  $\partial_{\nu_E} \phi=b_E^2  \, \tau_E\cdot [\ee]_E\tau_E$ on $E$ and  $\varepsilon(v)\perp \Curl^2\phi$,  
an inverse estimate shows
\[
 \| \tau_E\cdot[\ee]_E\tau_E\|^2_{L^2(E)}
 \lesssim (\ee-\varepsilon(v), \Curl^2 \phi)_{L^2(\omega_E)} - (\rot_{NC}^2\ee,\phi)_{L^2(\omega_E)}.
\]
At the heart of the bubble-function methodology are inverse and  trace inequalities 
that allow for appropriate scaling properties \cite{Verfuerth1996} under the overall assumption of shape-regularity. 
In the present case, one power of $h_E\approx h_{T_\pm}$ is hidden in the function $b$ and 
\begin{equation}\label{scalinglawsinlemma:eff:2}
 h_E^{1/2}\, 	| \phi |_{H^2(\omega_E)} 
+  h_E^{-3/2}	\| \phi \|_{L^2(\omega_E)}  \lesssim  \| \tau_E\cdot [\ee]_E\tau_E \|_{L^2(E)}.
\end{equation}
The combination with the previous estimate results in 
\[
 \| \tau_E\cdot[\ee]_E\tau_E\|^2_{L^2(E)}
 \hspace{-1mm}  \lesssim   \hspace{-1mm}  \| \tau_E\cdot [\ee]_E\tau_E \|_{L^2(E)}
  \hspace{-1mm} 
 \left(   \hspace{-1mm}  h_E^{-1/2}     \|\ee-\varepsilon(v)\|_{L^2(\omega_E)}     
  \hspace{-1mm} +  \hspace{-1mm}  
  h_E^{3/2}  \| \rot_{NC}^2\ee \|_{L^2(\omega_E)}\hspace{-1mm}\right) \hspace{-1mm} .
\]
This and Lemma \ref{lemma:eff:1} conclude the proof.
\end{proof}

For any boundary edge $E\in\E(\Omega)$,  the edge-bubble function 
$b_E=4\lambda_1\lambda_2\in W^{1,\infty}(\omega_E)$ 
for the two barycentric coordinates  $\lambda_1,\lambda_2$  associated with the 
two vertices of $E$ and  $b_E$
vanishes on the remaining sides $\partial\omega_E\setminus E$ of the aligned 
triangle $\overline{\omega_E}$.  The Dirichlet data $u_D$ allows for some polynomial approximation 
$\Pi_{m,E} u_{D}\in P_{m}(E)$ of a maximal  degree  bounded by $m\ge k+2$; recall the definition of
$\osc_I(u_D,E)$ from  \eqref{eqdefoscIondirichlet}.

\begin{lemma}[efficiency of first boundary edge residual]\label{lemma:eff:4}
Any function $v\in H^1(\omega_E;\R^2)$ with $v|_E=u_D|_E$ along   $E\in\E(\Gamma_D)$ satisfies 
\[
h_E^{1/2} \|\tau_E\cdot (   \ee\tau_E - \partial u_D/\partial s)\|_{L^2(E)}
\lesssim  \| \ee - \varepsilon(v) \|_{L^2(\omega_E)} + \osc_I(u_D,E).
\]
\end{lemma}

\begin{proof}
Since $\tau_E\cdot \ee\tau_E$ is a polynomial of degree at most $k+2\le m$ along the exterior edge $E$, the residual 
$\tau_E\cdot (   \ee\tau_E - \partial u_D/\partial s)$ is well approximated by its  $L^2$ projection 
$\rho_E:=  (\tau_E\cdot (   \ee\tau_E - \Pi_{m,E} \partial u_D/\partial s))$ onto $P_{m}(E)$.
The Pythagoras theorem based on the $L^2$ orthogonality reads
\[
h_E \|\tau_E\cdot (   \ee\tau_E - \partial u_D/\partial s)\|_{L^2(E)}^2=
h_E \| \rho_E\|_{L^2(E)}^2+\osc^2_I(u_D,E)
\]
and it remains to bound $h_E^{1/2} \| \rho_E\|_{L^2(E)}$ by the right-hand side of the claimed 
inequality.  The extension $P_E\rho_E\in C^\infty(\R^2)$  and 
the function $b$ from \eqref{testfunctioninlemma:eff:2}
lead to an admissible test function 
$
\phi:= b P_E \rho_E
\in W^{1,\infty}_0(\omega_E)$. Two successive integration by parts as in 
Lemma~\ref{lemma:representation} show
\[
(\varepsilon(v),\Curl^2\phi)_{L^2(\omega_E)}=(\partial u_D/\partial s ,\tau_E (\nu_E\cdot \nabla \phi))_{L^2(E)}.
\]
This and  Lemma~\ref{lemma:representation} lead to 
\begin{equation*}
  (\tau_E\cdot( \ee\tau_E- \frac{ \partial u_D}{\partial s}), \frac{\partial \phi}{\partial \nu_E} )_{L^2(E)}=
  (\ee-\varepsilon(v) , \Curl^2 \phi)_{L^2(\omega_E)} - (\rot_{NC}^2\ee,\phi)_{L^2(\omega_E)}.
\end{equation*}
Since   $\partial_{\nu_E} \phi = b_E^2  \rho_E$ along $E$ and $\rho_E$ is the $L^2$ projection 
of $\tau_E\cdot( \ee\tau_E- \partial u_D/\partial s)$, the left-hand side equals 
$\| b_E  \rho_E\|_{L^2(E)}^2 - ( ( 1-\Pi_{m,E})   \partial u_D/\partial s, b_E^2  \rho_E )_{L^2(E)}$.
The scaling argument which leads to  \eqref{scalinglawsinlemma:eff:2} shows that the left-hand side 
of \eqref{scalinglawsinlemma:eff:2} is  $\lesssim\| \rho_E\|_{L^2(E)}$. The combination with 
the previously displayed identity leads to
\[
 \| \rho_E\|_{L^2(E)}^2\lesssim  \| \rho_E\|_{L^2(E)} \hspace{-1mm} 
 \left(   \hspace{-1mm}  h_E^{-1/2}     \|\ee-\varepsilon(v)\|_{L^2(\omega_E)}     
  \hspace{-1mm} +  \hspace{-1mm}  h_E^{3/2}  \| \rot_{NC}^2\ee \|_{L^2(\omega_E)}
    \hspace{-1mm} +  \hspace{-1mm}h_E^{-1/2}\hspace{-1mm} \osc_I(E,u_D)   \hspace{-1mm} \right) \hspace{-1mm}.
\]
This and Lemma \ref{lemma:eff:1} conclude the proof.
\end{proof}

The edge-bubble functions for the  second edge residuals are defined slightly differently
to ensure some  vanishing normal derivative.

\begin{lemma}[efficiency of second interior edge residual]\label{lemma:eff:3}
Any $v\in H^1(\omega_E;\R^2)$,  $E\in\E(\Omega)$, satisfies
\[
  h_E^{3/2} \|\tau_E\cdot( [\rot_{NC} \ee]_E - \partial  [\ee]_E/\partial s\, \nu_E)\|_{L^2(E)} 
  \lesssim \| \ee - \varepsilon(v) \|_{L^2(\omega_E)}.
\]
\end{lemma}

\begin{proof}
There are many ways to define an edge-bubble function for this situation and one may first select a maximal 
open ball $B(x_E,2r_E) \subset \omega_E$ around  a point $x_E\in E$ with maximal radius $2r_E$, 
which is entirely included in $\omega_E$. The characteristic function $\chi_{B(x_E,r_E)}$ of the smaller ball
$B(x_E,r_E)$ may be regularized with a standard mollification $\eta_{r_E}$ to define the smooth function
$b:=\chi_{B(x_E,r_E)}*\eta_{r_E}\in C_c^\infty(\Omega_E)$ with values in $[0,1]$ 
and with $\nabla b\cdot \nu_E=0$ along $E$.  The polynomial 
$\rho_E:= \tau_E\cdot( [\rot_{NC}\ee]_E - \partial  [\ee]_E/\partial s \, \nu_E)$ 
and its extension  $P_E \rho_E$ define the  test function 
$\phi := b P_E \rho_E\in C_0^\infty(\omega_E)$ in 
Lemma~\ref{lemma:representation}. The representation formula 
and $(\varepsilon(v),\Curl^2\phi)_{L^2(\omega_E)}=0$ lead to
\[
\|  b^{1/2} \rho_E\|_{L^2(E)}^2
=  (\varepsilon(v)-\ee ,\Curl^2\phi)_{L^2(\omega_E)} + (\rot_{NC}^2\ee,\phi)_{L^2(\omega_E)} .
\]
The  inverse inequality  $\|  \rho_E\|_{L^2(E)}\|  \le \| b^{1/2} \rho_E\|_{L^2(E)}$, 
Cauchy inequalities, and the right  scaling properties of $\phi$ lead to 
\[
\|\rho_E\|_{L^2(E)}^2\lesssim \|\rho_E\|_{L^2(E)}\left(  
h_E^{-3/2} \|\ee-\varepsilon(v) \|_{L^2(\omega_E)} +    h_E^{1/2}\| \rot_{NC}^2\ee\| _{L^2(\omega_E)} \right). 
\]
This and Lemma \ref{lemma:eff:1} conclude the proof.
\end{proof}

The efficiency of the last edge contribution involves  
the  second Dirichlet data oscillation  $\osc_{II}(u_D,E)$ 
from  \eqref{eqdefoscIIondirichlet}.
 
\begin{lemma}[efficiency of second boundary edge residual]\label{lemma:eff:5}
Any function $v\in H^1(\omega_E;\R^2)$ with $v|_E=u_D|_E$ along   $E\in\E(\Gamma_D)$ satisfies 
\[
h_E^{3/2} \|\tau_E\cdot \rot \ee  - \nu_E\cdot (\frac{ \partial \ee \tau_E }{\partial s }
+\frac{ \partial^2 u_D}{ \partial s^2}) \|_{L^2(E)}
\lesssim  \| \ee - \varepsilon(v) \|_{L^2(\omega_E)} + \osc_{II}(u_D,E).
\]
\end{lemma}

\begin{proof}
Select a maximal  open ball $B(x_E,2r_E)\cap\Omega  \subset \omega_E$ around  a point $x_E\in E$ 
with maximal radius $2r_E$ such that  $B(x_E,2r_E)\cap \omega_E$ is a  half ball. The regularization 
$b:=\chi_{B(x_E,r_E)}*\eta_{r_E}\in C_c^\infty(\R^2)$ of the 
characteristic function 
$\chi_{B(x_E,r_E)}$  
attains values in $[0,1]$ and a
 positive integral mean   $h_E^{-1} \int_E b\, ds \approx 1$   along $E$ (depending
only on the shape regularity of $\T$); $b$  vanishes on $\partial\omega_E\setminus E$ 
and its normal derivative $\nabla b\cdot \nu=0$ vanishes 
along the entire boundary $\partial\omega_E$.

The Pythagoras theorem 
$  \|  \rho\|_{L^2(E)}^2 =  \|  \rho_E \|_{L^2(E)}^2 + h_E^{-3} \osc_{II}^2(u_D,E)$
for the residual 
$\rho:=\tau_E\cdot \rot \ee  - \nu_E\cdot (\frac{ \partial \ee \tau_E }{\partial s }
 +\frac{ \partial^2 u_D}{ \partial s^2})$ 
 and its  $L^2$ projection $\rho_E:=\Pi_{m,E}\rho$ onto $P_m(E)$
 reduces the proof to the estimation of  $\|  \rho_E \|_{L^2(E)}$.  The normal derivative 
 of $\phi:=  b\, P_E \rho_E\in C^\infty(\overline{\omega_E})$ vanishes along 
 the boundary $\partial\omega_E$ and  Lemma~\ref{lemma:representation} shows 
 \[ 
 (\rot\ee -\frac{ \partial \ee\nu_E}{\partial s} , b \rho_E  \tau_E)_{L^2(E)}
= (\rot_{NC}^2\ee,\phi)_{L^2(\omega_E)}-(\ee, \Curl^2 \phi)_{L^2(\omega_E)}.
 \]
The arguments in Lemma~\ref{lemma:bd} show  $  (\partial^2  u_D/\partial s^2   ,  b\rho_E \,  \nu_E)_{L^2(E)}
= (\varepsilon(v), \Curl^2 \phi)_{L^2(\omega_E)}$. The combination of the two identities leads to a formula for
$ (\rho,b\rho_E) _{L^2(E)}$. Since $\rho-\rho_E$ is controlled in $\osc_{II}^2(u_D,E)$, 
this and  an inverse inequality in the beginning result in 
\begin{align*}
&\|\rho_E\|_{L^2(E)}^2\lesssim  (b\rho_E,\rho_E) _{L^2(E)}
=(\rho, b\rho_E) _{L^2(E)}- (\rho-\rho_E, b\rho_E) _{L^2(E)} \\
&\; \lesssim (\rot_{NC}^2\ee,\phi)_{L^2(\omega_E)}
-(\ee-\varepsilon(v), \Curl^2 \phi)_{L^2(\omega_E)}  +\|\rho_E\|_{L^2(E) }h_E^{-3/2} \osc_{II}(u_D,E).
\end{align*}
The scaling properties of $\phi$ and its derivatives are as in the proof of the previous lemma. 
With Lemma \ref{lemma:eff:1} in the end, this  concludes the proof.
\end{proof}

%----------------------------------------------------------------------------------------------------
\section{Numerical examples}
\label{sec:examples}
%----------------------------------------------------------------------------------------------------
This section is devoted to numerical experiments for four different domains to 
demonstrate robustness in the reliability and efficiency of the a posteriori
error estimator $\eta(\T_\ell,\sigma_\ell)$. The implementation 
follows \cite{CGRT2008,CEG2011,CG_2016} for $k=1$ with Lam\'e 
 parameters $\lambda$ and $\mu$ from $\lambda= E\nu/( (1+\nu)(1-2\nu))$ 
 and $\mu={E}/(2(1+\nu))$ for a Young's modulus $E=10^5$ and various Poisson ratios 
 $ \nu=0.3$ and  $ \nu= 0.4999$.
%----------------------------------------------------------------------------------------------------
\subsection{Academic example}
%----------------------------------------------------------------------------------------------------
%
\begin{figure}[tbp]
\begin{center}
\includegraphics[width=0.8\textwidth]{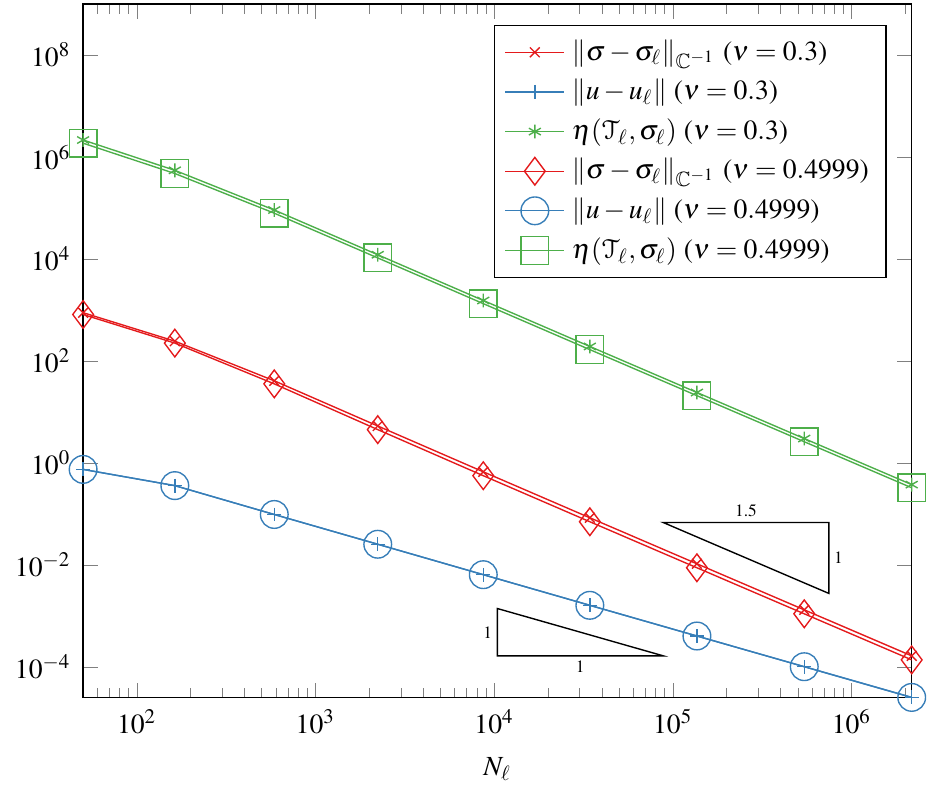}
\end{center}
\caption{Convergence history plot in  academic example}
\label{fig:Convergence1}
\end{figure}
The model problem \eqref{eq:lame} on
the unit square $\Omega = (0,1)^2 $ with homogeneous Dirichlet boundary conditions
and the right-hand side $f=(f_1,f_2)$, 
\[
f_1(x,y) = -f_2(y,x)= -2\mu\pi^3\cos(\pi y)\sin(\pi y)(2\cos(2\pi x) - 1)\quad\text{for }(x,y)\in\Omega,
\]
allows the smooth exact solution 
\[
u(x,y) =   \pi\,   \sin(\pi x)\sin (\pi y)\, \left(      \cos(\pi y)   \sin(\pi x), -\cos(\pi x)  \sin(\pi y)\right)
\quad\text{for }(x,y)\in\Omega.
\]
Note that $f$ depends only on the Lam\'e parameter $\mu$ and not on the critical Lam\'e parameter $\lambda$.
For  uniform mesh refinement,  Figure~\ref{fig:Convergence1}  displays the
robust third-order convergence of the 
a~posteriori error estimator $\eta(\T_\ell,\sigma_\ell)$ as well as the Arnold-Winther finite element stress error.
The convergence is robust    in the Poisson ratio $\nu\to 1/2$ and  the a~posteriori error estimator proves to be 
 reliable and efficient. In this example, the oscillations  $\osc(f,\T_\ell)$ dominate the a posteriori error estimator.
 
This typical observation motivates   numerical examples with $f\equiv 0$ in the sequel.

%----------------------------------------------------------------------------------------------------
\subsection{Circular inclusion}
%----------------------------------------------------------------------------------------------------
%
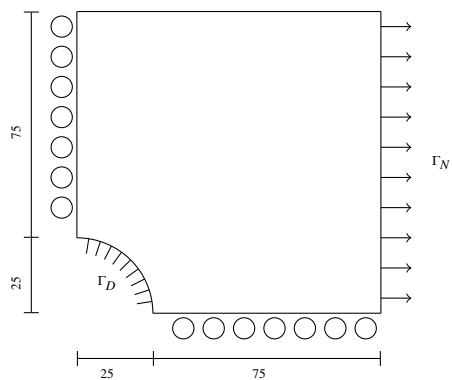
\begin{figure}[tbp]
\begin{center}
\begin{tikzpicture}[scale = 4]

\draw[-] (0.25,0) -- (1,0);
\draw[-] (1,0) -- (1,1);
\draw[-] (1,1) -- (0,1);
\draw[-] (0,0.25) -- (0,1);

\draw (0.25,0) arc (0:90:0.25);

\foreach \y in {0,...,9}
{
   \draw[->] (1,\y*0.1+0.05) -- ++(0.1,0);
}

\foreach \y in {3,...,9}
{
   \draw (-0.05,\y*0.1+0.05) circle (0.035);
}

\foreach \y in {3,...,9}
{
   \draw (\y*0.1+0.05,-0.05) circle (0.035);
}

\draw (0.1,0.1) node {\tiny$\Gamma_D$};

\draw (1.2,0.5) node {\tiny$\Gamma_N$};

\draw[|-] (-0.15,0) -- (-0.15,0.25);
\draw (-0.2,0.1) node[rotate=90] {\tiny 25};

\draw[|-|] (-0.15,0.25) -- (-0.15,1);
\draw (-0.2,0.6) node[rotate=90] {\tiny 75};

\draw[|-] (0,-0.15) -- (0.25,-0.15);
\draw (0.1,-0.2) node {\tiny 25};

\draw[|-|] (0.25,-0.15) -- (1,-0.15);
\draw (0.6,-0.2) node {\tiny 75};

\foreach \y in {1,...,9}
{
    \draw  (\y*9:0.25)  -- (\y*9:0.2);
}

\end{tikzpicture}
\end{center}
\caption{Domain circular inclusion}
\label{fig:ex:geom:CircularIncusion}
\end{figure}
\begin{figure}
\begin{center}
\includegraphics[width=0.8\textwidth]{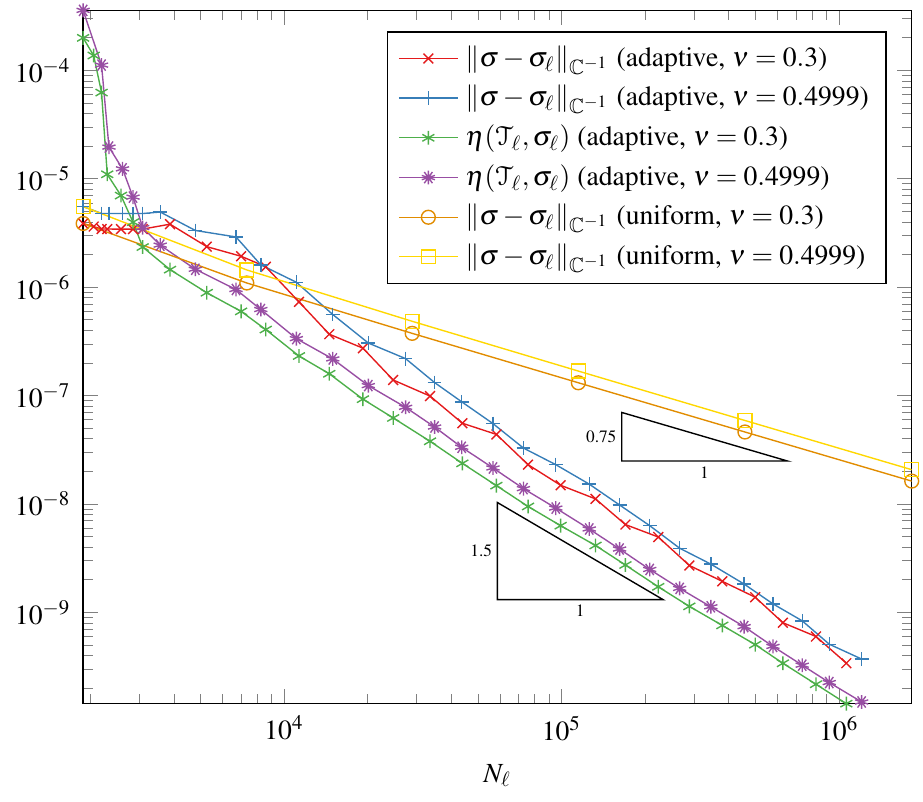}
\end{center}
\caption{Convergence history plot in circular inclusion benchmark }
\label{fig:Convergence2}
\end{figure}
The second benchmark example from the literature models a rigid circular inclusion in an infinite plate
for the domain $\Omega$   with rather mixed boundary conditions
indicated with mechanical symbols in Figure~\ref{fig:ex:geom:CircularIncusion}. 
The exact solution \cite{KS} to the model problem \eqref{eq:lame} reads (with polar coordinates $(r,\phi)$
and $\kappa = 3-4\nu$, $\gamma = 2\nu - 1$, $a=1/4$) 
\begin{align*}
	u_r &= \frac{1}{8\mu r}\left( (\kappa-1)r^2 + 2\gamma a^2 + 
	          \left( 2r^2 -\frac{2(\kappa+1)a^2}{\kappa}  + \frac{2 a^4}{\kappa r^2}\right) \cos(2\phi)\right),\\
  u_\phi &= -\frac{1}{8\mu r}\left(  2r^2-\frac{2(\kappa-1)a^2}{\kappa} - \frac{2a^4}{\kappa r^2}\right)\sin(2\phi).
\end{align*}
The  approximation of the  circular inclusion through a polygon is rather critical for the 
higher-order Arnold--Winther MFEM.
In the absence  of an implementation of parametric boundaries, adaptive mesh refinement
is necessary for higher improvements.  
The adaptive algorithm of this section is the same for all examples and acts on polygons; in particular, it  
does not monitor the curved 
boundary, but whenever some edge at the curved part $\Gamma_D$ is refined in this example,  the 
midpoint is a new node and  projected onto $\Gamma_D$. 
The convergence history plot in  Figure~\ref{fig:Convergence2} shows a reduced  convergence for uniform
refinement, while adaptive refinement (of the circular boundary) leads to optimal 
third-order convergence.

%----------------------------------------------------------------------------------------------------
\subsection{L-shaped benchmark}
%----------------------------------------------------------------------------------------------------
%
\begin{figure}[tbp]
\begin{center}
\begin{tikzpicture}[scale=1.25]
\path (0,0) coordinate (A);
\path (-1,-1) coordinate (B);
\path (0,-2) coordinate (C);
\path (2,0) coordinate (D);
\path (0,2) coordinate (E);
\path (-1,1) coordinate (F);

\draw (A) -- (B) -- (C) -- (D) -- (E) -- (F) -- (A);

\draw[->] (0,0) -- (2.5,0);
\draw[->] (0,0) -- (0,2.5);

\draw (2.5,-0.1) node {\tiny$x$};
\draw (-0.1,2.5) node {\tiny$y$};
\draw (1.3,1.3) node {\tiny$\Gamma_D$};

\draw (0,0) -- (0.5,0.5);
\draw (0.25,0.5) node {\tiny$r$};

\draw[->] (0.5,0) arc (0:45:0.5);
\draw (0.55,0.25) node {\tiny$\phi$};

\draw[|-|] (-1,-2.25) -- (0,-2.25);
\draw[-|] (0,-2.25) -- (2,-2.25);

\draw[|-|] (-1.25,-2) -- (-1.25,2);
\draw (-0.5,-2.4 ) node {\tiny 1};
\draw (1,-2.4) node {\tiny 2};
\draw (-1.4,0) node[rotate=90] {\tiny 4};

\draw (-0.5,0) node {\tiny $\Gamma_N$};

\foreach \y in {0,...,10}
{
    \draw  (-1+\y*0.1,-1-\y*0.1) -- ++(0,-0.15);
		\draw  (-1+\y*0.1,1+\y*0.1) -- ++(0,0.15);
}

\foreach \y in {0,...,20}
{
    \draw  (\y*0.1,-2+\y*0.1) -- ++(0,-0.15);
		\draw  (\y*0.1,2-\y*0.1) -- ++(0,0.15);
}
\end{tikzpicture}
\end{center}
\caption{L-shaped domain}
\label{fig:geom:Lshape}
\end{figure}
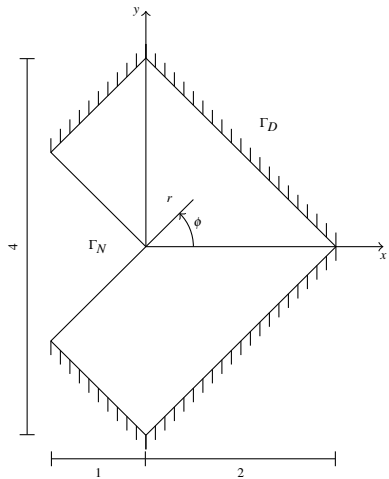
\begin{figure}
\begin{center}
\includegraphics[width=0.8\textwidth]{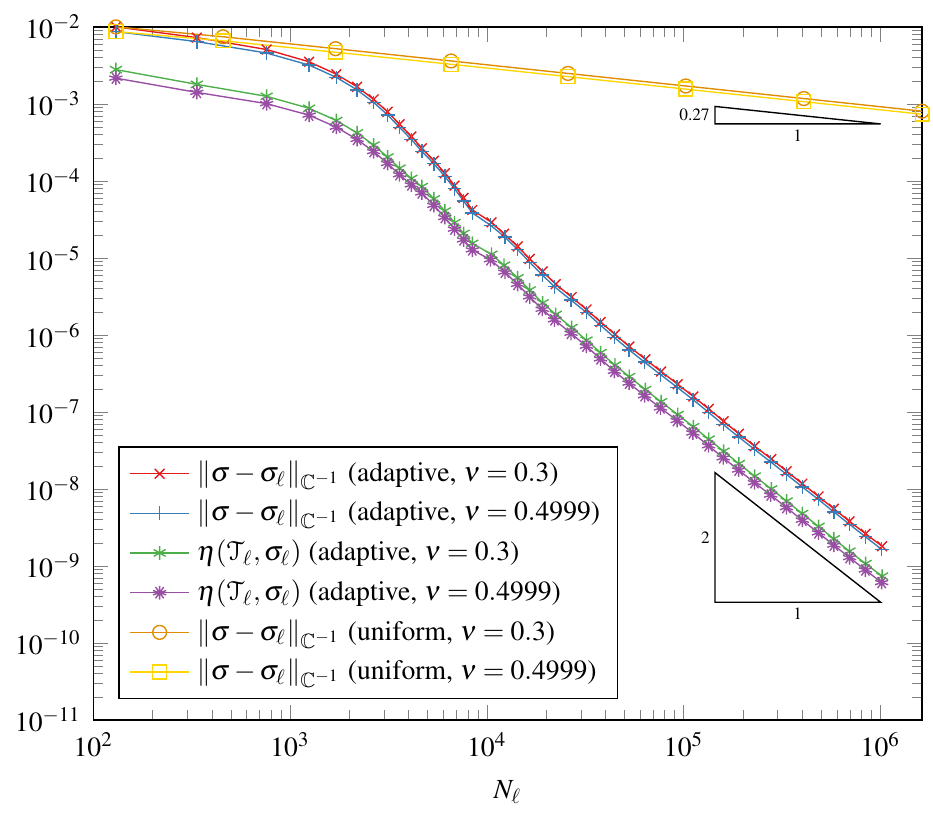}
\end{center}
\caption{Convergence history plot in L-shaped benchmark }
\label{fig:Convergence3}
\end{figure}
Consider the rotated L-shaped domain with Dirichlet and Neumann boundary depicted in Figure~\ref{fig:geom:Lshape}.
The exact solution reads in polar coordinates
\begin{align*}
 u_r(r,\phi) &= \frac{r^{\alpha}}{2\mu} \left( -(\alpha+1)\cos((\alpha+1)\phi) + (C_2-\alpha-1)C_1\cos((\alpha-1)\phi)\right),\\
 u_\phi(r,\phi) &= \frac{r^{\alpha}}{2\mu} \left( (\alpha+1)\sin((\alpha+1)\phi) + (C_2+\alpha-1)C_1\sin((\alpha-1)\phi)\right).
\end{align*}
The constants are $C_1 := -\cos((\alpha+1)\omega)/\cos((\alpha-1)\omega)$ and $C_2:= 2(\lambda + 2\mu)/(\lambda+\mu)$, 
where $\alpha= 0.544483736782$ is the first root of $\alpha\sin(2\omega)+\sin(2\omega\alpha)=0$ 
for $\omega = 3\pi/4$. The volume force $f\equiv 0$ and the Neumann boundary data 
$g\equiv 0$ vanish, and the Dirichlet boundary conditions $u_D$ are 
extracted from the exact solution.

Figure~\ref{fig:Convergence3} shows suboptimal convergence $\mathcal{O}(N_\ell^{-0.27})$, 
namely an expected rate 
$\alpha$ in terms of the maximal mesh-size,   for uniform  and fourth-order 
$L^2$ stress  convergence for adaptive mesh-refinement. 

Despite the singular solution, the adaptive algorithm recovers the higher convergence of 
Theorem~\ref{supperconvergence} as in \cite{CG_2016}.

%----------------------------------------------------------------------------------------------------
\subsection{Cook membrane problem}
%----------------------------------------------------------------------------------------------------
%
\begin{figure}[tbp]
\begin{center}
\begin{tikzpicture}[scale = 0.08]

\path (0,0) coordinate (A);
\path (0,44) coordinate (B);
\path (48,60) coordinate (C);
\path (48,44) coordinate (D);

\draw (D) -- (A) -- (B) -- (C);
\draw[dotted] (B) -- (D);

\foreach \y in {1,...,21}
{
   \draw (0,\y*2) -- ++(-3,-3);
}
\draw[->] (48,44) -- (48,52);
\draw[->] (48,52) -- (48,60);
\draw[->] (48,60) -- (48,65);

\draw (-6,22) node[fill=white] {\tiny$\Gamma_D\!\!$};
\draw (54,53) node {\tiny$\Gamma_N$};

\draw[|-|] (0,-5) -- (48,-5);
\draw (24,-8) node {\tiny 48};

\draw[|-|] (-10,0) -- (-10,44);
\draw[-|] (-10,44) -- (-10,60);
\draw (-15,22) node[rotate=90] {\tiny 44};
\draw (-15,52) node[rotate=90] {\tiny 16};

\draw (A) node[right] {\small A};
\draw (B) node[left] {\small D};
\draw (C) node[right] {\small C};
\draw (D) node[below] {\small B};

\end{tikzpicture}
\end{center}
\caption{Cook membrane}
\label{fig:geom:Cooks}
\end{figure}
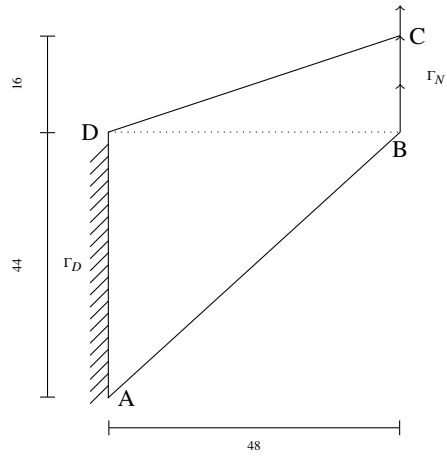
\begin{figure}
\begin{center}
\includegraphics[width=\textwidth]{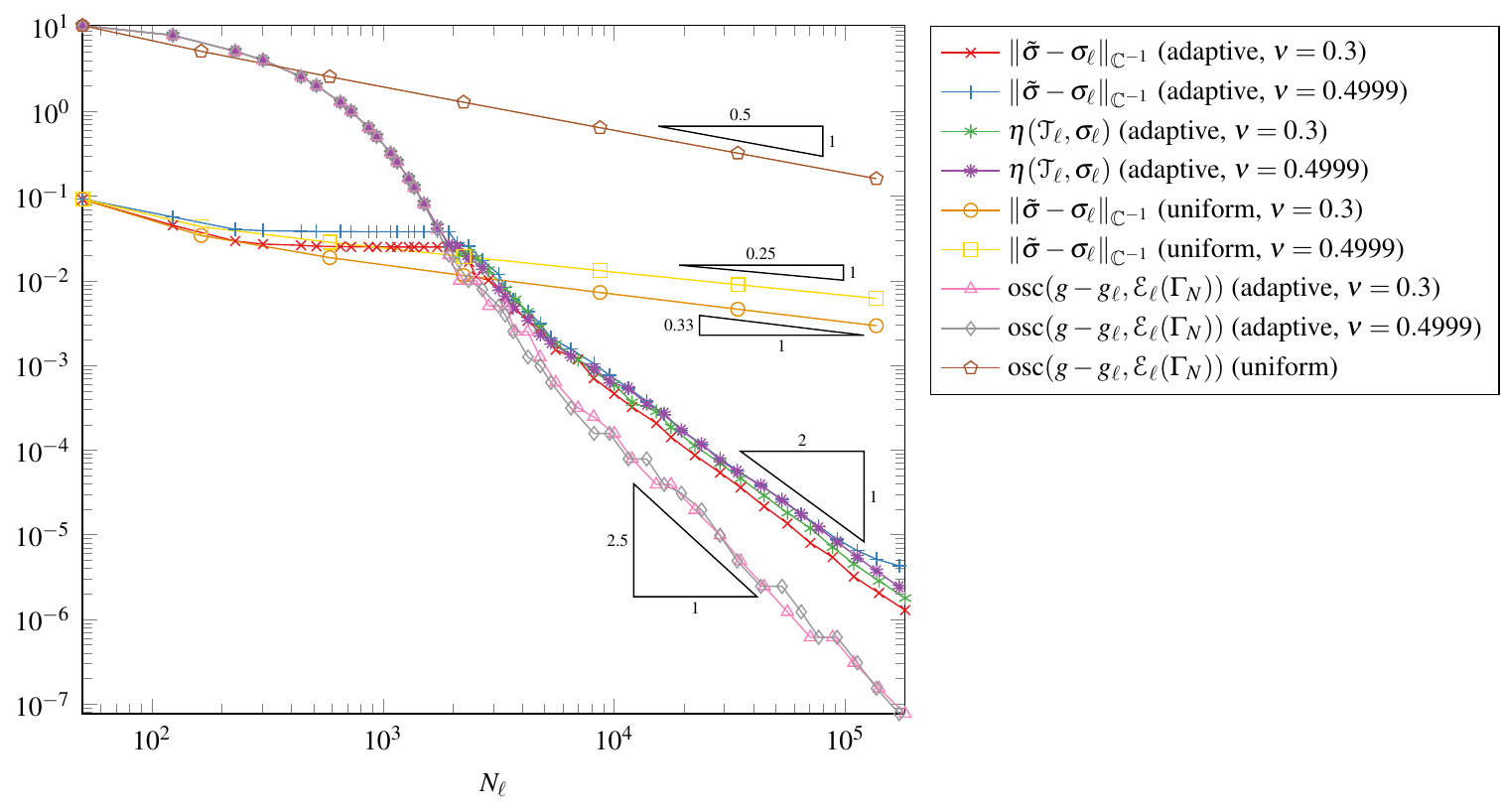}
\end{center}
\caption{Convergence history plot in Cook's membrane benchmark}
\label{fig:Convergence4}
\end{figure}
One of the more popular benchmarks in computational mechanics is 
the  tapered panel $\Omega$  with the vertices $A,B,C,D$ of 
Figure~\ref{fig:geom:Cooks} 
clamped on the left side $\Gamma_D=\operatorname{conv}(D,A)$ (with $u_D\equiv 0$)
under no volume force ($f\equiv 0$) but applied surface tractions 
$g = (0,1)$ along $\operatorname{conv}(B,C)$ and traction free on the remaining parts 
$\operatorname{conv}(A,B)$ and $\operatorname{conv}(C,D)$ 
along the Neumann boundary.

This example is a particular difficult one for the Arnold--Winther MFEM 
because of the incompatible Neumann boundary conditions on the right corners
\cite{CGRT2008,CEG2011,CG_2016}.
That means, although $g$ is piecewise constant, $g$ does not belong to $G(\T)$
for any triangulation.  In the two Neumann corner vertices $B$ and $C$ 
we therefore strongly impose the values
$\sigma_\ell(B) = (0.2491 , 0.7283;   0.7283 , 0.6676)$ and 
$\sigma_\ell(C) = ( 3/20, 11/20 ; 11/20 , 11/60)$ for the design of $g_\ell\in G(\T_\ell)$.

Since the exact solution is unknown, the error approximation rests on a reference solution $\tilde \sigma$ 
computed as $P_5(\T)$ displacement approximation on the uniform refinement of the finest adapted triangulation.

The large pre-asymptotic range of the convergence history plot in  Figure~\ref{fig:Convergence4}  illustrates the
difficulties of the Arnold-Winther finite element method  
in case of incompatible Neumann boundary conditions according to its
nodal degrees of freedom. Once the resulting and dominating boundary oscillations (caused by the necessary
 choice of discrete compatible Neumann conditions in $G(\T_\ell)$)  $\osc(g-g_\ell,\E_\ell(\Gamma_N))$  
 are resolved through adaptive mesh-refining, even the fourth-order $L^2$ stress convergence is visible 
 in a long asymptotic range 
in (the approximated error and)  the equivalent error estimator. 
   
 This example underlines that adaptive mesh-refining is unavoidable in computational mechanics with optimal rates 
 and a large saving in computational  time and memory compared to naive uniform mesh-refining.

\section*{Acknowledgements}
The work has been written, while the three authors enjoyed the hospitality of the 
Hausdorff  Research Institute of Mathematics in Bonn, Germany,  during the 
Hausdorff Trim\-ester Program {\em Multiscale Problems:
Algorithms, Numerical Analysis and Computation}.
The research of the first author (CC) has been supported by the
 Deutsche Forsch\-ungs\-gemeinschaft in the Priority Program 1748
`{\em Reliable simulation techniques in solid mechanics. Development of non-standard discretization
methods, mechanical and mathematical analysis}' under the project `{\em Foundation and application of
generalized mixed FEM towards nonlinear problems in solid mechanics}' (CA 151/22-1).
The second author (DG) has been supported by the
Deutsche Forschungsgemeinschaft (DFG) through CRC 1173.
The third author (JG) has been funded by the Austrian Science Fund (FWF) 
through the project P 29197-N32.

%----------------------------------------------------------------------------------------------------
% References
%----------------------------------------------------------------------------------------------------

\section*{Appendix: Fourth-order convergence of the stress in  $L^2$}
This appendix explains a high-order convergence phenomenon observed in some numerical benchmark examples for
the lowest-order Arnold-Winther method. Adopt the notation from this paper for $k=1$ and let 
$\sigma$ solve \eqref{eq:lame} and let $\sigma_h \in AW_k(\T)$ solve \eqref{eq:discpp}. 

\begin{theorem}[fourth-order convergence]\label{supperconvergence}
Suppose  $f=f_h\in P_1(\T;\R^2)$ and $g=g_h\in G(\T)$ and suppose that the stress
$\sigma\in H^4(\Omega;\mathbb{S})$. Then  
the $L^2$ stress error satisfies (with the maximal mesh-size $h$)
\[
\|\sigma-\sigma_h\|_{L^2(\Omega)} \lesssim h^4  \|\sigma\|_{H^4(\Omega)}.
\] 
\end{theorem}

\begin{proof}
Since the stress error $\sigma-\sigma_h\in H^4(\T;\mathbb{S})$ is divergence-free,
$\alpha$  vanishes in  \eqref{e:helmholtz} and $\sigma-\sigma_h =  \Curl^2 \beta\in H^4(\T;\mathbb{S})$.
Since $\beta\in H^2(\T)$ is piecewise in $C^2$, it follows $\beta\in C^1(\overline{\Omega})$.
The Arnold-Winther finite elements have nodal degrees of freedom at the vertices and hence 
$\sigma_h$ is continuous at each vertex $z\in\N$. Hence the second derivatives of $\beta\in C^2(\T)\cap C^1(\overline{\Omega})$
are continuous at each vertex $z\in\N$. It follows that the nodal interpolation operator $I_A$ associated to
the Argyris finite element space $ A(\T) \subset C^1(\Omega)\cap P_5(\T)$ exists for $\beta$
in the classical sense and is composed of the piecewise local interpolation. This defines 
$\beta_h= I_A\beta$ and the divergence-free $\tau_h:=\Curl^2\beta_h\in AW(\T)$ test function in 
 \eqref{eq:lame} and in \eqref{eq:discpp}. Consequently,
 \[
 \|\sigma-\sigma_h\|_{L^2(\Omega)}^2= (\sigma-\sigma_h , \Curl^2(\beta-\beta_h))_{L^2(\Omega)}
 \le  \|\sigma-\sigma_h\|_{L^2(\Omega)} \, |\beta-I_A\beta |_{H^2(\Omega)}.   
 \]
 This and standard local interpolation error estimates for the nodal interpolation of the 
 quintic  Argyris  finite elements \cite{BrennerScott2008,Braess2007,Cia78} 
 show
  \[
 \|\sigma-\sigma_h\|_{L^2(\Omega)}\lesssim h^4  |\beta |_{H^6(\T)}=h^4  |\sigma |_{H^4(\Omega)}.   
 \]
(With $\sigma-\sigma_h=\Curl^2\beta$ and $|\sigma_h |_{H^4(\T)} =0$ for piecewise cubic $\sigma_h$
in the last step.)  
\end{proof}
\end{document}